\documentclass[twocolumn]{autart}  
\usepackage{cite} 
\usepackage{color}
\usepackage{graphicx}         
\usepackage{amsmath}
\usepackage{subfigure}
\usepackage{amssymb}
\usepackage{amsfonts}
\usepackage{eucal}
\usepackage[titletoc,title]{appendix}
\usepackage{mathptmx} 
\usepackage{epsfig}
\usepackage{epstopdf}
\usepackage{subcaption} 

\usepackage{tikz} \usetikzlibrary{calc} \tikzset{>=latex}

\usepackage{bbm}
\usepackage{mathrsfs} 
\newtheorem{theorem}{Theorem}
\newtheorem{lemm}{Lemma}

\newenvironment{proof}[1][Proof]{\noindent\textbf{#1.} }{\ \rule{0.5em}{0.5em}}

\newtheorem{assump}{Assumption}
\usepackage{color} 
\allowdisplaybreaks
\usepackage{flushend}
\usepackage{verbatim} % comment 
\begin{document}      
	\setlength\abovedisplayskip{6pt}
	\setlength\belowdisplayskip{6pt}
	\setlength\abovedisplayshortskip{6pt}
	\setlength\belowdisplayshortskip{6pt}
	\allowdisplaybreaks
	\setlength{\parindent}{1em}
	\setlength{\parskip}{-0em}   
	\addtolength{\oddsidemargin}{6pt}      
	
	\begin{frontmatter}
		\title{Neural Operators for Adaptive Control of Freeway Traffic} 
		% Title, preferably not more 
		% than 10 words.
		
		\thanks[footnoteinfo]{Corresponding author: Huan Yu. This work is supported by the National Natural Science Foundation of China under Grant 62203131, 92471108 and 12131008.
		}	
		\author[Lv]{Kaijing Lv}\ead{kjlv@bit.edu.cn},  % Add the 
		\author[Lv]{Junmin Wang}\ead{jmwang@bit.edu.cn},
		\author[Yu]{Yihuai Zhang} \ead{yzhang169@connect.hkust-gz.edu.cn},
		\author[Yu]{Huan Yu*}\ead{huanyu@ust.hk} \ead{}

		\address[Lv]{School of Mathematics and Statistics,  Beijing Institute of Technology, Beijing,China, 100081} 
	\address[Yu]{Hong Kong University of Science and Technology (Guangzhou), Nansha, Guangzhou, Guangdong, China, 511400}


	\begin{abstract} 
The uncertainty in human driving behaviors leads to stop-and-go traffic congestion on freeway. The freeway traffic dynamics are governed by the Aw-Rascle-Zhang (ARZ) traffic Partial Differential Equation (PDE) models with unknown relaxation time. Motivated by the adaptive traffic control problem, this paper presents a neural operator (NO) based  adaptive boundary control design for the coupled 2$\times$2 hyperbolic systems with uncertain spatially varying in-domain coefficients and boundary parameter.  In traditional adaptive control for PDEs, solving backstepping kernel online can be computationally intensive, as it updates the estimation of coefficients at each time step. To address this challenge, we use operator learning, i.e. DeepONet, to learn the mapping from system parameters to the kernels functions. DeepONet, a class of deep neural networks designed for approximating operators, has shown strong potential for approximating PDE backstepping designs in recent studies.  Unlike previous works that focus on approximating single kernel equation associated with the scalar PDE system, we extend this framework to approximate PDE kernels for a class of the first-order coupled 2$\times$2 hyperbolic kernel equations. Our approach demonstrates that DeepONet is nearly two orders of magnitude faster than traditional PDE solvers for generating kernel functions, while maintaining a loss on the order of $10^{-3}$. In addition, we rigorously establish the system's stability via Lyapunov analysis when employing DeepONet-approximated kernels in the adaptive controller. The proposed adaptive control is compared with reinforcement learning (RL) methods. Our approach guarantees stability and does not rely on initial values, which is essential for rapidly changing traffic scenarios. This is the first time this operator learning framework has been applied to the adaptive control of the ARZ traffic model, significantly enhancing the real-time applicability of this design framework for mitigating traffic congestion.

	\end{abstract}
	
\end{frontmatter}

\section{Introduction}\label{sec:introduction}
Stop-and-go traffic congestion is a very common phenomenon in major cities around the world. The traffic congestion on highways leads to many unsafe driving behaviors, as well as increased fuel emissions, environmental pollution, and increased commuting time \cite{belletti2015prediction}\cite{de2011traffic}. The traffic congestion is characterized by the propagation of shock waves on road, caused by delayed driver response. There have been many studies on traffic stabilization using PDE models,  such as the first-order hyperbolic PDE model proposed by Ligthill and Whitham and Richards (LWR) \cite{lighthill1955kinematic}\cite{richards1956shock} to describe traffic density waves on highways. Then Aw and Rascle \cite{aw2000resurrection} and Zhang \cite{zhang2002non} proposed the second-order nonlinear hyperbolic PDE model to describe the evolution of velocity and density states in traffic flow. The ARZ model is a 2$\times$2 hyperbolic PDE system and widely used for describing dynamics of the stop-and-go traffic oscillations. In this paper, we adopt the ARZ model and develop adaptive boundary control designs for traffic stabilization. 
\subsection{PDE backstepping for traffic control}
\par The control strategy for freeway traffic congestion is usually based on static road infrastructure to regulate traffic flow, such as ramp metering and varying speed limits. Various traffic boundary control designs have been proposed to smooth traffic in the works of Bekiaris-Liberis and Delis \cite{bekiaris2019feedback}, Zhang \cite{zhang2019pi} as well as Karafyllis, Bekiaris-Liberis, and Papageorgiou \cite{karafyllis2018feedback}. While Bekiaris-Liberis and Delis utilize Adaptive Cruise Control vehicles for in-domain actuation as control inputs \cite{bekiaris2019feedback}, Karafyllis et al. design a boundary feedback law to manage inlet demand \cite{karafyllis2018feedback}. The boundary control strategy using PDE backstepping is first proposed in \cite{yu2019traffic} to stabilize the linearized ARZ system, including full state feedback and output feedback. Recent efforts \cite{burkhardt2021stop,yu2018adaptive,yu2022traffic,yu2018varying,yu2022simultaneous,zhang2023mean} have further developed backstepping controllers for various traffic scenarios including multi-lane, multi-class and mixed-autonomy traffic.
This paper primarily focuses on adaptive control of traffic PDE systems with uncertain parameters.  

\par In traffic flow modeling, relaxation time is a critical parameter representing drivers' reaction delays to evolving traffic conditions. However,  heterogeneity and unpredictability of individual driver behavior makes it impossible to obtain the relaxation time in practice. 
This uncertainty in relaxation time can significantly impact the stability and performance of traffic systems. Traditional control methods struggle to handle such uncertainties, making it difficult to ensure system stability and optimal performance under varying traffic conditions. To address these challenges, we adopt adaptive control strategies that allow for real-time adjustment of the controller gains to accommodate unknown or time-varying system characteristics, ensuring the desired system performance. 

\textcolor{black}{
Early developments in adaptive control for PDEs, as in \cite{logemann1997adaptive}, addressed systems stabilized via high-gain feedback, under a relative degree one condition. While these approaches ensured parameter  identifiability, they required control input to be applied throughout the spatial domain. Considerable progress has been achieved in the adaptive stabilization of PDEs with uncertain parameters, especially for hyperbolic and parabolic systems\cite{bohm1998model,hong1994direct,krstic2008adaptive,smyshlyaev2007adaptive}}. Adaptive control methods \cite{di2014adaptive,belhadjoudja2023adaptive,kawan2022lyapunov} can be categorized into Lyapunov-based design, identifier-based design and swapping-based design. After a decade of research, advancements in adaptive control have begun to be applied to coupled hyperbolic PDEs \cite{anfinsen2019adaptive}.

\par Although adaptive control for PDE systems with unknown parameters has been extensively studied~\cite{krstic2008backstepping,smyshlyaev2010boundary,anfinsen2018adaptive,hu2015control,wang2020event,auriol2020output} and was first applied for the ARZ PDE model in~\cite{yu2018adaptive}. The practical implementation of the adaptive controller for the traffic systems still faces challenge. This is because the adaptive control process simultaneously requires the estimation of unknown system parameters and PDE states. After each time step, it is necessary to recalculate the solution to the PDE corresponding to the gain kernel function in order to update the estimated system parameter functions. This places extremely high demands on real-time computation. The computational resources required for calculation of the gain function increase significantly with spatial sampling precision when applying traditional finite difference and finite element methods. In this paper, we adopt neural operators to accelerate computation of adaptive PDE backstepping controllers. 

\subsection{Advances in machine learning for PDE traffic control}

With rapid advances in machine learning, data-driven methods for solving,  modeling and control of PDEs have received widespread attention including physics-informed learning, reinforcement learning and operator learning. Physics-Informed Neural Networks (PINNs) directly incorporates physical constraints into neural networks training by embedding the physical laws of PDEs into the loss function. This enables PINNs to solve PDEs without large amounts of training data. Mowlavi and Nabi extend PINNs method to PDE optimal control problems in \cite{mowlavi2023optimal}. Zhao proposed a novel hybrid Traffic state estimation (TSE) approach called Observer-Informed Deep Learning (OIDL), which integrates a PDE observer and deep learning paradigm to estimate spatial-temporal traffic states from boundary sensing data in \cite{zhao2023observer}. However, PINNs need to be retrained for each new set of boundary and initial conditions, which poses limitations in adaptive control applications.

Reinforcement learning (RL) has also been increasingly applied for PDE control problems, particularly in boundary and feedback control. RL continuously optimizes strategies to achieve real-time control of complex PDE systems. In the field of traffic management, researchers have been applying RL to various traffic issues. Wu et al. used the city mobility traffic micro-simulator SUMO to design a deep RL framework for hybrid autonomous traffic in various experimental scenarios \cite{wu2017emergent}. Under the same framework, \cite{qu2020jointly} proposed a reinforcement learning-based car-following model for electric, connected, and automated vehicles to reduce traffic oscillations and improve energy efficiency. \cite{yu2021reinforcement} presented the exploration using RL for traffic PDE boundary control.
However, RL has limited generalization ability in practical applications. RL
may perform well under the specific initial conditions. However, for initial conditions outside the training range, there may be performance degradation or even failure. RL may be sensitive to hyperparameters and exhibit unpredictable behavior, making it difficult to ensure consistent and stable performance in different scenarios.

\par Traditional neural networks typically learn mappings between finite dimensional Euclidean spaces, but with the advancement of research, this method has been extended to the field of NO~\cite{lu2021learning}. NO-based learning methods focus on mapping between function spaces and are specifically designed for solving PDEs and dynamical systems. Compared with traditional machine learning methods, NO have two unique advantages. Firstly, theoretically speaking, NO can learn the mapping of the entire system parameter set, rather than being limited to a single system parameter like standard neural networks. Secondly, from an empirical perspective, research work~\cite{lu2021learning}\cite{shi2022machine} has shown that NO have significantly better accuracy than traditional deep learning methods when simulating complex functions. Therefore, NO not only solves individual equation instances, but can also handle the problems of the entire PDE family.

Recent research has effectively utilized DeepONet for one-dimensional transport PDEs\cite{bhan2023neural}, reaction-diffusion equations and observer designs\cite{krstic2024neural}, as well as for  hyperbolic PDEs  with delay \cite{qi2024neural}, parabolic PDEs with delays \cite{wang2025deep}, 2$\times$2 hyperbolic PDEs \cite{wang2025backstepping}, traffic flow \cite{zhang2024neural} and cascaded parabolic PDEs \cite{lv2024neural}. 
In contrast to the approximate backstepping transformations used in \cite{bhan2023neural,krstic2024neural,qi2024neural,wang2025deep,wang2025backstepping,zhang2024neural,lv2024neural}, this paper adopts the exact backstepping transformation, referred to as a gain-only approach. The gain-only approach focuses on approximating a 1D gain kernel, simplifying network design, reducing training set size and time, and easing the derivation of the perturbed target system, which have been successfully used in gain scheduling that adjust controller gains based on current states of nonlinear PDE system\cite{lamarque2024gain} and several benchmark unstable PDEs\cite{vazquez2024gain}. A recently developed method based on power series approximations \cite{vazquez2023power}, along with its MATLAB extension \cite{lin2024towards}, shows promise as a tool for generating training datasets. The application of NO-approximated gain kernels becomes even more valuable for adaptive control, where the kernel must be recomputed online at each time step to accommodate updated estimates of the plant parameters.  This was first explored for first-order hyperbolic PDE in \cite{lamarque2025adaptive} and extended to the reaction-diffusion equation in \cite{bhan2025adaptive}. Different from \cite{lamarque2025adaptive}\cite{bhan2025adaptive}, where the kernel equation involves a single kernel, in this work, we extend the results of \cite{lamarque2025adaptive} to the ARZ traffic models which involved the coupled heterogeneous hyperbolic PDEs. The technical challenges arise from both the more complex kernel computations and the proof analysis of the higher-order PDE systems with the approximated controllers. 

\textbf{\emph{Contributions:}}
The main contributions are summarized as follows:
\begin{itemize}	
	\item[$\bullet$]We present an NO-based adaptive control method to stabilize the ARZ traffic PDE model with unknown relaxation time. Additionally, we extend stability schemes for more general 2$\times$2 hyperbolic systems with uncertain spatially varying in-domain coefficients and boundary parameter. Compared to the relevant works \cite{lamarque2025adaptive}\cite{bhan2025adaptive}, which approximate single kernel, a key technical challenge is dealing with the approximation of coupled 2$\times$2 Goursat-form PDE kernels in the stabilization of coupled 2$\times$2 hyperbolic PDEs.
	\item[$\bullet$]To address the computational challenges associated with solving gain kernel equations, we integrate DeepONet into the adaptive control framework.  
	It is shown that the NO is almost two orders of magnitude faster than the PDE solver in solving kernel functions, and the loss remains on the order of $10^{-3}$. 
	To the best of our knowledge, this is the first study to integrate DeepONet with adaptive control in traffic flow systems, demonstrating its potential to improve the computational efficiency of control schemes in congested traffic scenarios.
	\item[$\bullet$] Through comparative experiments with RL, it has been proven that our method does not rely on initial values compared to RL and provides a model-based solution with guaranteed stability. In addition, we theoretically prove the system's stability through Lyapunov analysis when replacing with the DeepONet approximation kernels in the adaptive controller.
\end{itemize}
\textbf{\textit{Organization of paper:}}The paper is organized as follows. 
Section \ref{sec_Nominal} introduces ARZ traffic PDE model and a nominal adaptive backstepping control scheme designed for  2$\times$2 hyperbolic PDEs. Section \ref{sec_Approximation} gives a series of properties for the gain kernel and its time derivative and introduces the approximation of feedback kernel operators.
Section \ref{sec_DeepONet} presents the stabilization achieved through the application of approximate controller gain functions via DeepONet. Numerical simulations are presented in Section \ref{sec_Simulation}. Section \ref{sec_Conclusion} presents the conclusion.

\vspace{-10pt}
\paragraph*{Notation.}
\begin{table}[]
	\centering
	\begin{tabular}{|l|c|}
		\hline
		exact operator & $\mathcal K$ 
		\\ \hline 
		neural operator & $\hat{\mathcal K}$
		\\ \hline%\hline
		unknown model parameters & $({c}_1,{c}_2,{c}_3,{c}_4,{r})$
		\\ \hline 
		estimated model parameters & $(\hat{c}_1,\hat{c}_2,\hat{c}_3,\hat{c}_4,\hat{r})$
		\\ \hline 
		exact kernel & $(K^u,K^m)=\mathcal K ({c}_1,{c}_2,{c}_3,{c}_4,{r})$
		\\ \hline
		exact estimated kernel & $(\breve{K}^u,\breve{K}^m) = \mathcal K(\hat{c}_1,\hat{c}_2,\hat{c}_3,\hat{c}_4,\hat{r})$
		\\ \hline
		approximate estimated kernel  & 
		$(\hat{K}^u,\hat{K}^m) = \hat{\mathcal K}(\hat{c}_1,\hat{c}_2,\hat{c}_3,\hat{c}_4,\hat{r})$ 
		\\ \hline
	\end{tabular}
	\caption{Nomenclature for kernel learning with exact and approximate operators}
	\label{tab:nomenclature}
\end{table}

We present the nomenclature for kernel learning with exact and approximate operators in Table 1. We define the $L^2$-norm for $\chi(x)\in L^2[0,1]$ as $\rVert\chi \rVert_2=\int_0^1|\chi(x)|^2dx$. \textcolor{black}{We use $\|\cdot \|_\infty$ for the infinity-norm, that is $\|\chi\|_\infty=\sup_{x \in [0,1]} |\chi(x)|$.} We set $\rVert\chi \rVert_1=\int_0^1 |\chi(x)|\,dx$. 

\section{{Nominal Adaptive Control Design}}\label{sec_Nominal}
\subsection{ARZ PDE Traffic Model}
The ARZ PDE model is used to describe the formation and dynamics of the traffic oscillations which refer to variations of traffic density and speed around equlibrium values. It consists of a set of 2$\times$2 hyperbolic PDEs for traffic density and velocity. The ARZ model of $(\rho(x,t),v(x,t))$-system is given by
\begin{equation}\label{eq:ARZ_1}
\begin{aligned}
	\partial_{t} \rho+\partial_{x}(\rho v) & =0,  \\
	\partial_{t} (v-V(\rho))+v\partial_{x}\left(v-V(\rho)\right)  & =\frac{V(\rho)-v}{\tau} ,\\
	\rho(0, t) & =\frac{q^{*}}{v(0, t)}, \\
	v(L, t) & =U(t)+v^{*}, 
\end{aligned}
\end{equation}
where $(x,t)\in [0,L]\times \mathbb{R}_+$, 
$\rho(x, t)$ represents the traffic density, $v(x, t)$ represents the traffic speed, and $\tau$ denotes the relaxation time, which refers to the time required for driver behavior to adapt to equilibrium. This parameter is used to describe the process by which vehicle speed adjusts to match the traffic density. \textcolor{black}{The variable $p(\rho)$, defined as the traffic system pressure, is related to the density by the equation}
\begin{equation*}
	 \textcolor{black}{p(\rho)=c_0(\rho)^{\gamma},}
\end{equation*}
 \textcolor{black}{and $c_0,\gamma \in \mathbb{R}_+$.} The equilibrium velocity-density relationship $V (\rho)$ is given in Greenshield model:
\begin{equation*}
	V(\rho)=v_{f}\left(1-\left(\frac{\rho}{\rho_{m}}\right)^\gamma\right),
\end{equation*}
where $v_f$ is the velocity of free flow, $\rho_m$ is the maximum density of free flow, and $(\rho^*,v^*)$ are the equilibrium points of the system with $v^*=V(\rho^*)$. \textcolor{black}{We consider a constant traffic flux $q^* = \rho^* v^*$ entering the domain from $x = 0$ and there is a Varying Speed Limit(VSL) boundary control at the outlet. $U(t)$ is defined as variation from steady state velocity. The VSL at outlet shows $v^*$ with $U(t)$ which we will design later.} We can apply the change of coordinates introduced in \cite{yu2018varying} to rewrite it in the Riemann coordinates and then map it to a decoupled first-order 2$\times$2 hyperbolic system.
\begin{equation}
\label{eq:linearized_ARZ}
    \begin{aligned} 
	\partial_{t} {u}_1+v^{*}\partial_{x} {u}_1 & =0, \\
	\partial_{t} {m}_1-( \gamma p^*-v^{*}) \partial_{x}{m}_1 & ={c}(x) {u}_1, \\
	{u}_1(0, t) & =r_0 {m}_1(0, t), \\
	{m}_1(L, t) & =U(t), 
\end{aligned}
\end{equation}
where 
\begin{align*}
	{c}(x)&=-\dfrac{1}{\tau}\exp(-\dfrac{x}{\tau v^*}),r_0=\frac{\rho^* V(\rho^* )+v^*}{v^*}.
\end{align*}

The relaxation time $\tau$ describes how fast drivers adapt their speed to equilibrium speed-density relations. Its value is usually difficult to measure in practice and is easily affected by various external factors. Therefore, we propose adaptive control law. Motivated by the second-order ARZ model, we first propose NO-based adaptive design for a more general framework of 2$\times$2 hyperbolic PDEs with spatially varying coefficients, as the linearized ARZ model \eqref{eq:linearized_ARZ} is a special case of such systems. 
\subsection{{Adaptive Control for Coupled 2$\times$2 Hyperbolic PDEs}}
\par We consider the first-order coupled 2$\times$2 hyperbolic PDE system with four spatially variable coefficients, 
\begin{equation}
\label{eq:uv_def}
    \begin{aligned}
	\partial_{t} u(x, t)  +\lambda \partial_{x} u(x, t)& =c_{1}(x) u(x, t)+c_{2}(x) m(x, t), 
    \\
	\partial_{t} m(x, t)-\mu \partial_{x} m(x, t) & =c_{3}(x) u(x, t)+c_{4}(x) m(x, t), 
    \\
	u(0, t) & =r m(0, t), 
    \\
	m(1, t) & =U(t), %\label{eq:uv_def_v0}
\end{aligned}
\end{equation}
where $t \in \mathbb{R}_{+}$ is the time, $x \in [0, 1]$ is the space, the states
are given by  $u,m$ and the initial conditions are $u(x,0)=u_0(x)$, $m(x,0)=m_0(x)$ where $u_0,m_0 \in L^2([0,1])$. The positive transport speeds $\lambda, \mu \in \mathbb{R}$ are known. We assume the spatially variable coefficients $c_{1}(x),c_{2}(x),c_{3}(x),c_{4}(x) \in C([0,1])$ and boundary coefficient $r \in \mathbb{R}$ are unknown.  
	\par Note that system \eqref{eq:uv_def} is 2$\times$2 hyperbolic system with spatially variable coefficients in domain, which is different from the system in \cite{anfinsen2018adaptive} with the constant coefficients. System \eqref{eq:uv_def} is a direct extension of the system in \cite{anfinsen2018adaptive}, where the difference lies in the designed adaptive update law.

\par To ensure the well-posedness of the kernel PDEs, the adaptive control estimation requires bounded assumptions. Our basic assumption is as follows.
\begin{assump}
	 Bounds are known on all uncertain parameters, that is, there exists some constants $\bar{c}_i$, $i=1\cdots 4$, and $\bar{r}$ so that 
	\begin{equation}
		\|c_i\|_{\infty} \leq \bar{c}_i,  i=1\cdots 4, |r|\leq \bar{r}.\label{eq:bounds}
	\end{equation}
\end{assump}

\par We first propose an adaptive control design using passive identifier design method, which includes the exact estimated backstepping kernels $\breve{K}^u,\breve{K}^m$.

\par We consider the identifier 
\begin{equation}
\label{eq:uv_identifier_u}
    \begin{aligned}
	\partial_{t} \hat{u}(x, t) =&-\lambda \partial_{x}\hat{u}(x, t)+\hat{c}_{1}(x,t)u(x,t)   \\
	&+\hat{c}_{2}(x,t)m(x,t)+\rho e_1(x,t)\|\varpi(t)\|^{2},  
    \\
	\partial_{t} \hat{m}(x, t) =&\mu \partial_{x}\hat{m}(x, t) +\hat{c}_{3}(x,t)u(x,t)   \\
	&+\hat{c}_{4}(x,t)m(x,t)+ \rho e_2(x, t)\|\varpi(t)\|^{2}, 
    \\
	\hat{u}(0, t)  =& \frac{\hat{r} u(0, t)+u(0, t) m^{2}(0, t)}{1+m^{2}(0, t)}, 
    \\
	\hat{m}(1, t)  =& U(t),
\end{aligned}
\end{equation}
where $\rho>0$,
\begin{equation}
	e_1(x,t) = u(x,t) - \hat{u}(x,t), e_2(x,t) = m(x,t) - \hat{m}(x,t),
	\label{eq:e_identifier_def}
\end{equation}
are errors between $u$ and $m$ and their estimates $\hat{u}$ and $\hat{m}$
, $\hat{c}_i$ and $\hat{r}$ are estimates $c_i$ and $r$. We define 
\begin{equation*}
	\varpi(x, t)=[u(x, t),m(x, t)]^{T},
\end{equation*}
for some initial conditions
$$\hat{u}_{0}, \hat{m}_{0} \in L^2([0,1]).$$
The error signals \eqref{eq:e_identifier_def} can straightforwardly be shown to have dynamics
\begin{align}
	\partial_{t} e_1(x, t) =&-\lambda \partial_{x}e_1(x, t)+\tilde{c}_{1}(x,t)u(x,t)  \nonumber \\
	&+\tilde{c}_{2}(x,t)m(x,t)-\rho e_1(x, t)\|\varpi(t)\|^{2},  \label{eq:e_1}\\
	\partial_{t} e_2(x, t) =&\mu \partial_{x}e_2(x, t) +\tilde{c}_{3}(x,t)u(x,t)  \nonumber \\
	&+\tilde{c}_{4}(x,t)m(x,t)- \rho e_2(x, t)\|\varpi(t)\|^{2}, \\
	e_1(0, t)  =&\frac{\tilde{r}(t) m(0, t)}{1+m^{2}(0, t)}, \label{eq:e_3}\\
	e_2(1, t) =&0, \label{eq:e_4}
\end{align}
where
\begin{equation*}
    \tilde{r}=r - \hat{r},\quad \tilde{c}_i=c_i - \hat{c}_i, \quad i =1,\cdots, 4.
\end{equation*}
We choose the following update laws
\begin{equation}
\label{eq:lawc12}
    \begin{aligned}
	\hat{c}_{1t}(x,t) & =\operatorname{Proj}_{\bar{c}_{1}}\left\{\gamma_{1}  e^{-\gamma x}e_1(x, t)u(x, t) , \hat{c}_{1}(x,t)\right\}, 
    \\
	\hat{c}_{2t}(x,t) & =\operatorname{Proj}_{\bar{c}_{2}}\left\{\gamma_{2}  e^{-\gamma x}e_1(x, t)m(x, t), \hat{c}_{2}(x,t)\right\}, 	
    \\
	\hat{c}_{3t}(x,t) & =\operatorname{Proj}_{\bar{c}_{3}}\left\{\gamma_{3}  e^{\gamma x}e_2(x, t)u(x, t) , \hat{c}_{3}(x,t)\right\}, 
    \\
	\hat{c}_{4t}(x,t) & =\operatorname{Proj}_{\bar{c}_{4}}\left\{\gamma_{4}  e^{\gamma x}e_2(x, t)m(x, t), \hat{c}_{4}(x,t)\right\}, 	
    \\
	\dot{\hat{r}}(t) & =\operatorname{Proj}_{\bar{r}}\left\{\gamma_{5}e_1(0, t) m(0, t), \hat{r}(t)\right\},
\end{aligned}
\end{equation}
where $ \gamma,\gamma_1, \gamma_2,\gamma_{3},\gamma_{4},\gamma_{5}>0$ are scalar design gains. \textcolor{black}{$\operatorname{Proj}$ denotes the projection operator}
\begin{align*} %\nonumber
	\textcolor{black}{\operatorname{Proj}_{\bar{\omega}}\{\tau,\hat{\omega}\}=\left\{
	\begin{array}{rcl}
		0 ~~~~    &    & {|\hat{\omega}|\geq \bar{\omega} ~\mbox{and}~\hat{\omega} \tau \geq 0},\\
		\tau ~~~~     &    & \mbox{otherwise}.
	\end{array} \right.}
\end{align*}
\textcolor{black}{The projection operator \(\operatorname{Proj}_{\bar{\omega}}\{\tau,\hat{\omega}\}\) is designed to constrain the update of parameter \(\hat{\omega}\), ensuring that it does not exceed the predefined bound \(\bar{\omega}\).}
The adaptive laws \eqref{eq:lawc12} have the following properties for all $t>0$ \cite{krstic2009delay}
\begin{align}\label{eq:projection_proper}
	-\tilde{\omega}^T\operatorname{Proj}_{\bar{\omega}}\{\tau,\hat{\omega}\} \leq -\tilde{\omega}^T\tau.
\end{align}
\begin{lemm}\label{lemma:Identifier_bound}
	{\em [Properties of passive identifier ]}
	Consider the system \eqref{eq:uv_def} and the identifier \eqref{eq:uv_identifier_u}, with an arbitrary initial condition $\hat u_0 = \hat u(\cdot,0), \hat m_0 = \hat m(\cdot,0)$ such that $\|\hat u_0\|<\infty,\|\hat m_0\|<\infty$, along with the update law \eqref{eq:lawc12} with an 
	arbitrary Lipschitz initial conditions satisfying the bounds \eqref{eq:bounds},
	guarantees the following properties
	\begin{align}
		\|\hat{c}_{i}(\cdot, t)\|_{\infty} \leq \bar{c}_i,\quad \forall i=1,\cdots,4,\quad |\hat{r}| \leq& \bar{r} , \label{properties_1} 
        \\
		\|e_1\| , \|e_2\| \in& L^{\infty} \cap L^2 ,
        \\
		\|e_1\|\|\varpi\|,\|e_2\|\|\varpi\| \in& L^2 , \label{properties_3}
        \\
		|e_1(0, \cdot)|, |e_1(1, \cdot)|, |e_2(0, \cdot)|,|e_1(0,\cdot) u(0, \cdot)| \in& L^2, \label{properties_4}
        \\
		\|{\partial_t \hat{c}_{i}}\|,|\dot{\hat{r}} |\in& L^2,
        \\
		\dfrac{\tilde{r}m(0,\cdot)}{\sqrt{1+m^2(0,\cdot)}} \in& L^2.\label{properties_5}
	\end{align}
\end{lemm}
\begin{proof}
	The proof can be found in the Appendix A.
\end{proof}
\par Considering the plant \eqref{eq:uv_def} with  unknown parameters $c_i, i=1,\cdots,4$ and $r$, we will design a nominal adaptive control law to achieve global stability. 

\par We consider the following adaptive backstepping transformation
\begin{equation}
\label{eq:backstepping_1}
    \begin{aligned}
	w(x,t)=&\hat{u}(x,t),
    \\
	z(x,t)=&\hat{m}(x,t)-\int_{0}^{x}\breve{K}^{u}(x,\xi,t)\hat{u}(\xi,t)d\xi 
    \\
	&- \int_{0}^{x}\breve{K}^{m}(x,\xi,t)\hat{m}(\xi,t)d\xi = T[\hat{u},\hat{m}](x,t), %\label{eq:backstepping_2}
\end{aligned}
\end{equation}
where the kernels $\breve{K}^{u}$ and $\breve{K}^{m}$ satisfy the following kernel functions
\begin{equation}
\label{eq:kernel_1}
    \begin{aligned}
	\mu \breve{K}_{x}^{u}(x, \xi, t)= &\lambda \breve{K}_{\xi}^{u}(x, \xi, t)+\hat{c}_{3}(\xi,t)\breve{K}^{m}(x, \xi, t)
    \\
	&+(\hat{c}_{1}(\xi,t)-\hat{c}_{4}(\xi,t)) \breve{K}^{u}(x, \xi, t), 
    \\
	\mu\breve{K}_{x}^{m}(x, \xi, t)= &- \mu \breve{K}_{\xi}^{m}(x, \xi, t)+ \hat{c}_{2}(\xi,t) \breve{K}^{u}(x, \xi, t), 
    \\
	\breve{K}^{u}(x, x, t)= & -\frac{\hat{c}_{3}(x,t)}{\lambda+\mu}, 
    \\
	\breve{K}^{m}(x, 0, t)= & \dfrac{\lambda \hat{r}(t)}{\mu}\breve{K}^{u}(x, 0, t).%\label{eq:kernel_4}
\end{aligned}
\end{equation}
The coupled \(2\times2\) Goursat-form PDEs, governed by two gain kernels, are defined over the triangular domain \(\mathcal{T}_1\), given by:
\begin{align}
	\mathcal{T}_1 = \{ (x, \xi) \; | \; 0 \leq \xi \leq x \leq 1 \}.
\end{align}
\textcolor{black}{The kernel functions \(\breve{K}^{u}(x, \xi, t)\) and \(\breve{K}^{m}(x, \xi, t)\) are computed online by solving the equation of \(2 \times 2\) hyperbolic PDEs whose dynamics depend on the unknown parameters \(\hat{c}_i(x,t)\), \(i = 1, \dots, 4\), and \(\hat{r}(t)\).} \textcolor{black}{These parameters are continuously estimated and updated via the adaptive laws \eqref{eq:lawc12}, so that at each time step, the kernel functions must be continually recalculated according to the new \(\hat{c}_i(x,t)\) and \(\hat{r}(t)\)
estimate.} 

\par Using the transformation \eqref{eq:backstepping_1}, we get the following target system
\begin{equation}
\label{eq:target_1}
    \begin{aligned}
	w_t(x,t) =&-\lambda w_x(x,t)+ \hat{c}_{1}w(x,t)+\hat{c}_{1}e_1(x,t)+\hat{c}_{2}z(x,t) 
    \\
	&+\int_{0}^{x}\omega(x,\xi,t)w(\xi,t)d\xi 
    \\
	&+ \int_{0}^{x}\kappa(x,\xi,t)z(\xi,t)d\xi  
    \\
	&+\hat{c}_{2}e_2(x,t)+\rho e_1(x,t)\|\varpi(t)\|^2, 
    \\
	z_{t}(x, t)= & \mu z_{x}(x, t)+\hat{c}_{4}z(x,t)- \lambda\breve{K}^{u}(x, 0, t) r(t) e_2(0, t) 
    \\
	&- \lambda \breve{K}^{u}(x, 0, t) \tilde{r}(t) z(0, t)+ \lambda\breve{K}^{u}(x, 0, t) e_1(0, t) 
    \\
	& -\int_{0}^{x} \breve{K}_{t}^{u}(x, \xi, t) w(\xi, t) d \xi 
    \\
	&-\int_{0}^{x} \breve{K}_{t}^{m}(x, \xi, t) T^{-1}[w, z](\xi, t) d \xi 
    \\
	& +T\left[\hat{c}_{1}e_1+\hat{c}_{2}e_2, \hat{c}_{3} e_1+\hat{c}_{4}e_2 \right](x, t)
    \\
	&+\rho T[e_1, e_2](x, t)\|\varpi(t)\|^{2}, \\
	w(0, t)= & r(t) z(0, t)+r(t) e_2(0, t)-e_1(0, t), \\
	z(1, t)= & 0, 
\end{aligned}
\end{equation}
where the coefficient $\omega$ and $\kappa$ are chosen to satisfy
\begin{align}
	\omega(x, \xi, t) & =\hat{c}_{2}(x,t) \breve{K}^{u}(x, \xi, t)+\int_{\xi}^{x} \kappa(x, s, t) \breve{K}^{u}(s, \xi, t) d s, \nonumber\\
	\kappa(x, \xi, t) & =\hat{c}_{2}(x,t) \breve{K}^{m}(x, \xi, t)+\int_{\xi}^{x} \kappa(x, s, t) \breve{K}^{m}(s, \xi, t) d s .\nonumber
\end{align}
From the boundary condition of \eqref{eq:uv_def}, \eqref{eq:backstepping_1} and boudary contion of \eqref{eq:target_1}, the nominal stabilizing	controller is straightforwardly derived as follows
\begin{equation}\label{U_exact}
	U(t)=\int_{0}^{1}\breve{K}^u(1,\xi,t)\hat{u}(\xi,t)d\xi + \int_{0}^{1}\breve{K}^m(1,\xi,t)\hat{m}(\xi,t)d\xi .
\end{equation}
Next, we present the stability of exact adaptive backstepping control, which serves as a guide for what we aim to achieve under the \textit{NO-based approximate adaptive backstepping} design. 
\begin{theorem}\label{theorem1}
	{\em [Stability of exact adaptive backstepping control]} Consider the plant \eqref{eq:uv_def} in feedback with the adaptive control law \eqref{U_exact} along with the update law for $\hat{c}_{1},\hat{c}_{2},\hat{c}_{3},\hat{c}_{4},\hat{r}$ given by \eqref{eq:lawc12} and the passive identifier $\hat{u},\hat{m}$ given by \eqref{eq:uv_identifier_u} satisfies the following properties for all solutions for all time:
	\begin{align*}
		&\|u\|,\|m\|,\|\hat{u}\|,\|\hat{m}\|, \|{u}\|_{\infty},\|{m}\|_{\infty}, \|\hat{u}\|_{\infty},\|\hat{m}\|_{\infty} \in L^2\cap L^{\infty},\\
		& \|u\|,\|m\|, \|\hat{u}\|,\|\hat{m}\|, \|\hat{u}\|_{\infty},\|\hat{m}\|_{\infty},\|\hat{u}\|_{\infty},\|\hat{m}\|_{\infty} \mapsto 0.
	\end{align*}
\end{theorem}
\begin{proof}
	The proof of this theorem follows a similar approach to the proof of Theorem 9.1 in \cite{anfinsen2019adaptive}.
\end{proof}

\begin{figure}[!t]
	\centering{\includegraphics[scale=0.55]{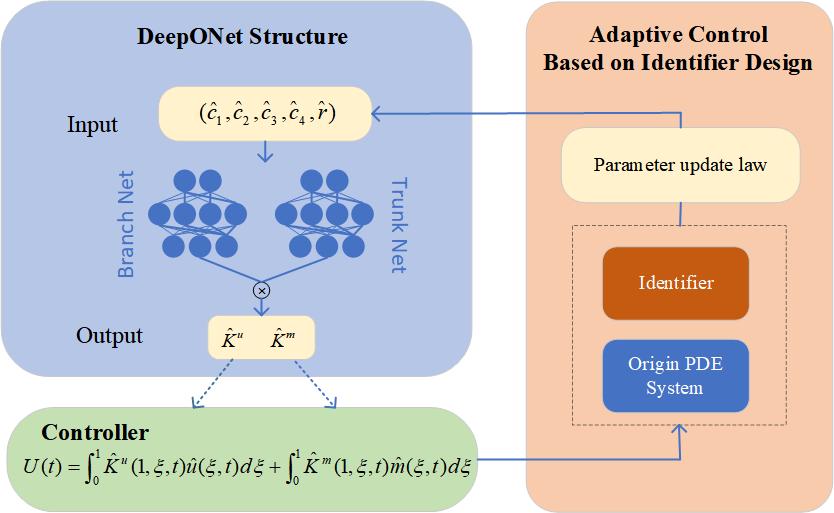}}
	\caption{The operator learning framework for adaptive control. }
	\label{DeepONet-System_2}
\end{figure} 

In summary, the exact adaptive backstepping feedback law \eqref{U_exact} can achieve global stability at the equilibrium point, with the system states $u(x,t),m(x,t)$ converging pointwise to zero. However, this method is computationally intensive because it requires solving Volterra equations \eqref{eq:kernel_1} at each time step t. To simplify the computation, we propose using the NO $\hat{\mathcal{K}}: (\hat{c}_1,\hat{c}_2,\hat{c}_3,\hat{c}_4,\hat{r}) \mapsto (\hat{K}^u,\hat{K}^m)$ to approximate the exact adaptive backstepping gain operator $\mathcal{K}: (\hat{c}_1,\hat{c}_2,\hat{c}_3,\hat{c}_4,\hat{r}) \mapsto (\breve{K}^u,\breve{K}^m)$. This approach allows for neural network evaluation at each time step instead of solving the complex equations. In the following section, we will introduce how to approximate the operator $\mathcal{K}$  using DeepONet and use the resulting approximate control gain functions for boundary stabilization of plant \eqref{eq:uv_def}. And we will use the universal approximation theorem of DeepONet \cite{Approximation} to derive the stability theorem from control gain kernel approximations.

\section{Neural Operator for Approximating Gain Kernels}\label{sec_Approximation}
By proving the continuity and boundedness mentioned above, it can be further deduced that for a set of continuous coefficients within a certain supremely bounded norm, there exists a NO with arbitrary accuracy. 

\subsection{Properties of the gain kernel functions}
\begin{lemm}\label{lem1-kernel}  
	Let $\|\hat{c}_{i}(\cdot, t)\|_{\infty} \leq \bar{c}_i, |\hat{r}| \leq \bar{r}$, $ \forall (x, t) \in [0, 1] \times \mathbb{R}_+$. Then, for any fixed $t\in\mathbb{R}_+$ and for any $\lambda,\mu \in \mathbb{R}_+$, $\hat{c}_i \in C([0,1])$ and $r \in \mathbb{R}_+$, $t\geq0$,
	the gain kernels $\breve{K}^{u},\breve{K}^{m}$ satisfying the PDE systems \eqref{eq:kernel_1}, have unique $C(\mathcal{T}_1)$ solutions with the property 
	\begin{align}
		|\breve{K}^{u}(x,\xi,t)|\leq& \bar{K}, \label{equ-ku-bouded}
        \\
		|\breve{K}^{m}(x,\xi,t)|\leq& \bar{K}, 
		\label{equ-k-bouded} 
        \\
		\|\breve{K}^u_{t}(x,\xi,t)\| \leq & M_{1}\|\hat{c}_{1t}\|+M_{2}\|\hat{c}_{2t}\|+M_{3}\|\hat{c}_{3t}\| \nonumber
        \\
		& +M_{4}\|\hat{c}_{4t}\|+M_{5}|\dot{\hat{r}}|,\label{equ-kut-bouded}
        \\
		\|\breve{K}^m_{t}(x,\xi,t)\| \leq & M_{6}\|\hat{c}_{1t}\|+M_{7}\|\hat{c}_{2t}\|+M_{8}\|\hat{c}_{3t}\|  \nonumber 
        \\
		& +M_{9}\|\hat{c}_{4t}\|+M_{10}|\dot{\hat{r}}|,  \label{equ-kvt-bouded}  
	\end{align}
	where $\bar{K}>0,  M_{j}>0,   j=1,\cdots,10$ are constants depending on the parameter bounds \eqref{eq:bounds}.
	
	\begin{proof}
        Fix time \( t \geq 0 \). The existence of a unique, bounded solution to kernel equations \eqref{eq:kernel_1} are guaranteed by \cite[Theorem A.1]{coron2013local}. Furthermore, this result ensures that the solution satisfies uniform bounds of the form 
\begin{equation*}
	|\breve{K}^{u}(x,\xi,t)|\leq \bar{K},
	|\breve{K}^{m}(x,\xi,t)|\leq \bar{K}, \quad t \geq 0,
\end{equation*}
where \( \bar{K} \) is a constant determined by the compactness of the admissible parameter set $\hat{c}_1,\cdots,\hat{c}_4,r$. 

Differentiating equation \eqref{eq:kernel_1} with respect to time yields the following system in terms of \( \breve{K}^u_t \) and \( \breve{K}^m_t \):
\begin{equation}
\label{eq:B2a}
    \begin{aligned}
		\mu\,  \breve{K}^u_{tx} - \lambda \breve{K}^u_{t\xi} =& \left( \hat{c}_1 - \hat{c}_4 \right) \breve{K}^u_t + \ {\hat{c}}_{3}\breve{K}^m_t 
        \\
		&+ \left( {\hat{c}}_{1t} - {\hat{c}}_{4t}\right) \breve{K}^u +\hat{c}_{3t} \breve{K}^m,  
        \\
		\mu\breve{K}^m_{tx} + \mu\breve{K}^m_{t\xi} =& \hat{c}_2\, \breve{K}^u_t + {\hat{c}}_{2t}\, \breve{K}^u,  
        \\
		\breve{K}^m_t(x, 0) =& \frac{\lambda \hat{r}}{ \mu}\, \breve{K}^u_t(x, 0) + \frac{\lambda\dot{\hat{r}}}{\mu}\, \breve{K}^u(x, 0),  
        \\
		\breve{K}^u_t(x, x) =& - \frac{{\hat{c}}_{3t}}{\lambda + \mu}. 
	\end{aligned}
\end{equation}
Applying again the result of~\cite[Theorem A.1]{coron2013local} to the equations~\eqref{eq:B2a}, one obtains the existence and uniqueness of a bounded solution \( (\breve{K}^u_t, \breve{K}^m_t) \), with bounds of the form \eqref{equ-kut-bouded} -\eqref{equ-kvt-bouded}.

	\end{proof}
\end{lemm}

\subsection{Approximation of the neural operator}
In the following discussion, we will first introduce the universal approximation theorem of DeepONet. This theorem demonstrates DeepONet's ability to approximate operators, enabling us to use it to learn the backstepping gain kernel mapping of PDEs and theoretically guarantee the stability of the adaptive control system.
\textcolor{black}{
\begin{theorem}\label{TheoDeepONet0}
	(DeepONet universal approximation theorem\cite{Approximation}). Let $X\subset \mathbb{R}^{d_x}$ and $Y\subset \mathbb{R}^{d_y}$ be compact sets of vectors $x \in X$ and $y \in Y$, respectively. Let $\mathcal{U}:X \mapsto U \subset \mathbb{R}^{d_u}$ and $\mathcal{V}:Y \mapsto V \subset \mathbb{R}^{d_v}$ be sets of continuous functions $u(x)$ and $m(y)$, respectively. Let $\mathcal{U}$ be also compact. Assume the operator $\mathcal{G}:\mathcal{U} \mapsto \mathcal{V}$ is continuous. Then for all $\epsilon >0$, there exist $n^*, p^* \in \mathbb{N}$ such that for each $n\geq n^*,p\geq p^*$, there exist $\theta^{(i)}, \vartheta^{(i)}$ for neural networks $f^{\mathcal{N}}(\cdot;\theta^{(i)}), g^{\mathcal{N}}(\cdot;\vartheta^{(i)}), i=1,\dots,p$, and $x_j \in X,j=1,\dots,n$, with corresponding $\mathbf{u}_n=(u(x_1),u(x_2),\dots,u(x_n))^T$, such that 
	\begin{equation*}
		|\mathcal{G}(u)(y)-\mathcal{G}_{\mathbb{N}}(\mathbf{u}_n)(y)|\leq \epsilon
	\end{equation*}
	where 
	\begin{equation*}
		\mathcal{G}_{\mathbb{N}}(\mathbf{u}_n)(y)=\sum\limits_{i=1}^{p} g^{\mathcal{N}}(\mathbf{u}_n;\vartheta^{(i)})f^{\mathcal{N}}(y;\theta^{(i)}),
	\end{equation*}
	for all functions $u\in \mathcal{U}$ and for all values $y\in Y$ of $\mathcal{G}(u)$.
\end{theorem}}
\iffalse
\begin{theorem}\label{TheoDeepONet0}
	[Universal approximation theorem of DeepONet\cite{Approximation}].For sets \( X \subset \mathbb{R}^{d_x} \) and \( Y \subset \mathbb{R}^{d_y} \), they are compact sets of vectors \( x \in X \) and \( y \in Y \), respectively. Define \( \mathcal{U}: X \mapsto U \subset \mathbb{R}^{d_u} \) and \( \mathcal{V}: Y \mapsto V \subset \mathbb{R}^{d_v} \) as sets of continuous functions \( u(x) \) and \( m(y) \), respectively, and assume that \( \mathcal{U} \) is also a compact set. If the operator \( \mathcal{G}: \mathcal{U} \mapsto \mathcal{V} \) is continuous, then for any \( \epsilon  > 0 \), there exist \( a^* \) and \( b^* \in \mathbb{N} \) such that when \( a \geq a^*, b \geq b^* \), there exist neural networks \( f^{\mathcal{N}}(\cdot; \theta^{(i)}), g^{\mathcal{N}}(\cdot; \theta^{(i)}), i=1, \dots, b \) with parameters \( \theta^{(i)}, \vartheta^{(i)} \), and their corresponding \( x_j \in X, j=1, \dots, a \), such that \( \mathbf{u}_a = (u(x_1), u(x_2), \dots, u(x_a))^T \) satisfies the following conditions.

	\begin{equation}
		|\mathcal{G}(u)(y)-\mathcal{G}_{\mathbb{N}}(\mathbf{u}_a)(y)|\leq \epsilon ,
	\end{equation}
	where 
	\begin{equation}
		\mathcal{G}_{\mathbb{N}}(\mathbf{u}_a)(y)=\sum\limits_{i=1}^{b} g^{\mathcal{N}}(\mathbf{u};\vartheta^{(i)})f^{\mathcal{N}}(y;\theta^{(i)}),
	\end{equation}
	for all $u\in \mathcal{U}$ and $y\in Y$ of $\mathcal{G}(u)\in \mathcal{V}$.
\end{theorem}
\fi

\par Figure \ref{DeepONet-System_2} presents a schematic diagram of the control circuit, illustrating the utilization of neural operators to accelerate the generation process of gain kernel functions in PDE adaptive control. 

\textcolor{black}{As illustrated in Figure~\ref{structure}, the DeepONet consists of two subnetworks: a Branch Net and a Trunk Net. The Branch Net takes a five-channel 2D input and extracts spatial features using two convolutional layers, each with a kernel size of 5 and a stride of 2. The output is flattened and passed through two fully connected layers with 512 and 256 neurons, respectively, each followed by a ReLU activation. The Trunk Net receives spatial coordinate input and processes it through three fully connected layers with 884, 128 and 256 neurons, also activated by ReLU functions. The outputs from both networks are then combined via inner product to produce the approximated kernel function valu $\hat{\mathcal{K}}$.}

\begin{figure}[t]
	%\centerline{\includegraphics[width=\textwidth,height=9pc,draft]{Openloop_uw.png}}
	\centerline{\includegraphics[scale=0.55]{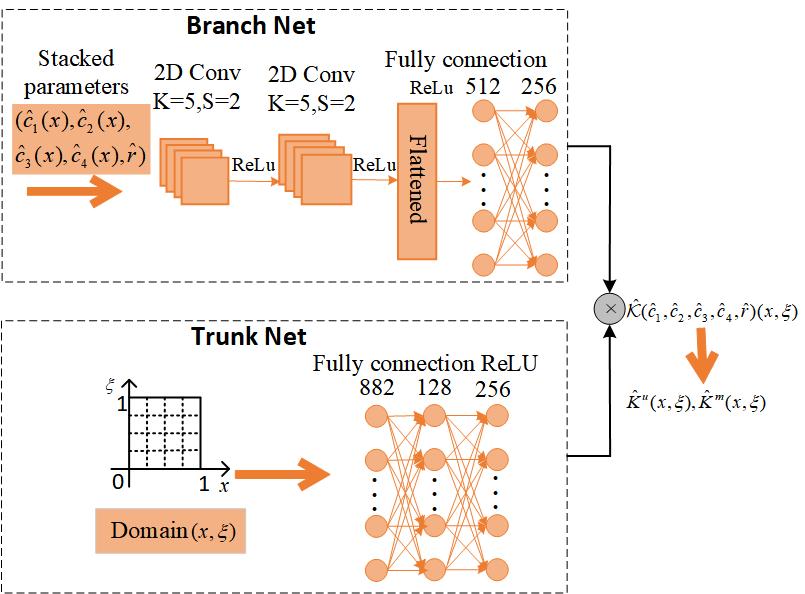}}
	\caption{ The DeepONet structure for operator $\hat{\mathcal{K}}$\label{structure}}
	%Open-loop system state $u(x,t)$ and $u(x,t)$
\end{figure}

\begin{theorem}\label{th:NO}
	{\em [Existence of a neural operator to approximating the kernels]} %Consider the neural operator defined in \eqref{eq:operatorK}, along with \eqref{eq:kernel_1}.
	Fix $t>0$. Fix a compact set $K \subset (C([0, 1]))^4 \times\mathbb{R}$ and define the operator $\mathcal{K}: K \mapsto (C(\mathcal{T}_1))^2$
\begin{align*}
	\mathcal{K}(\hat{c}_1,\hat{c}_2,\hat{c}_3,\hat{c}_4,\hat{r}) (\cdot,t):=(\breve{K}^u(x,\xi),\breve{K}^m(x,\xi)).
\end{align*}
Then, for all $\epsilon> 0$, there exists a neural operator $\mathcal{\hat{K}}: K \mapsto (C(\mathcal{T}_1))^2$ such that \textcolor{black}{for all $ (x,\xi) \in \mathcal{T}_1$}, 
	\begin{align*}
		|\mathcal{K}(\hat{c}_1,\hat{c}_2,\hat{c}_3,\hat{c}_4,\hat{r})(\cdot,t)-\hat{\mathcal{K}}(\hat{c}_1,\hat{c}_2,\hat{c}_3,\hat{c}_4,\hat{r})(\cdot,t)|\leq\epsilon.
	\end{align*}
\end{theorem}
\begin{proof}
	The continuity of the operator \(\mathcal{K}\) is derived from Lemma \ref{lem1-kernel}. And this result is based on Theorem 2.1 proposed by B. Deng et al. in their study \cite{Approximation}.
\end{proof}

\section{ Stabilization under DeepONet-Approximated Gain Feedback}\label{section_DeepONet}\label{sec_DeepONet}
We will demonstrate that although the adaptive controller uses approximate estimated kernel functions, the stability of the system is still guaranteed. Based on Theorem \ref{theorem1}, we present the system stability proof for the adaptive backstepping controller using approximate estimation kernels $\hat{K}^u,\hat{K}^m$  in the following theorem.
\begin{theorem}
	{\em [Stabilization under approximate adaptive backstepping control]} For all $\bar{c}_1,\bar{c}_2,\bar{c}_3,\bar{c}_4,\bar{r}>0$, there exists a constant $\epsilon_0>0$ such that for all NO approximations $\hat{K}^u,\hat{K}^m$ of accuracy $\epsilon  \in (0,\epsilon_0)$ provided by Theorem \ref{th:NO}, the plant \eqref{eq:uv_def} in feedback with the adaptive control law 
	\begin{equation}\label{eq:U_NO}
		U(t)=\int_{0}^{1}\hat{K}^u(1,\xi,t)\hat{u}(\xi,t)d\xi + \int_{0}^{1}\hat{K}^m(1,\xi,t)\hat{m}(\xi,t)d\xi,
	\end{equation} 
	along with the update law for $\hat{c}_{1},\hat{c}_{2},\hat{c}_{3},\hat{c}_{4}$ and $ \hat{r}$ given by \eqref{eq:lawc12} with any Lipschitz initial condition $\hat{c}_{10}=\hat{c}_{1}(\cdot,0),\hat{c}_{20}=\hat{c}_{2}(\cdot,0),\hat{c}_{30}=\hat{c}_{3}(\cdot,0),\hat{c}_{40}=\hat{c}_{4}(\cdot,0)$ such that $\|\hat{c}_{10}\|\leq \bar{c}_1, \|\hat{c}_{20}\|\leq \bar{c}_2, \|\hat{c}_{30}\|\leq \bar{c}_3, \|\hat{c}_{40}\|\leq \bar{c}_4$, and $\textcolor{black}{|\hat{r}(0)|}\leq \bar{r}$. The passive identifier $\hat{u}$, $\hat{m}$ given by \eqref{eq:uv_identifier_u} with any initial condition $\hat{u}_0=\hat{u}(\cdot,0),\hat{m}_0=\hat{m}(\cdot,0)$ such that $\|\hat{u}_0\|<\infty,\|\hat{m}_0\|<\infty$, the following properties hold:
	\begin{align*}
		&\|u\|,\|m\|,\|\hat{u}\|,\|\hat{m}\|, \|{u}\|_{\infty},\|{m}\|_{\infty}, \|\hat{u}\|_{\infty},\|\hat{m}\|_{\infty} \in L^2\cap L^{\infty},\\
		& \|{u}\|_{\infty},\|{m}\|_{\infty},\|\hat{u}\|_{\infty},\|\hat{m}\|_{\infty} \mapsto 0.
	\end{align*}
	Moreover, for the equilibrium $(u,m,\hat{u},\hat{m},\hat{c}_{1},\hat{c}_{2},\hat{c}_{3},\hat{c}_{4},\hat{r})=(0,0,0,0,{c}_{1},{c}_{2},{c}_{3},{c}_{4},r)$ the following global stability estimate holds 
	\begin{equation}\label{eq:S(t)}
		S(t)\leq 2\dfrac{k_2}{k_1}\theta_2 S(0) e^{\theta_1 k_2S(0)}, \quad t>0,
	\end{equation}
	where
	\begin{align*}
		S(t) :=& \|u\|^2+\|m\|^2+\|\hat{u}\|^2+\|\hat{m}\|^2+\|\tilde{c}_{1}\|^2+\|\tilde{c}_{2}\|^2 \nonumber \\
		&+\|\tilde{c}_{3}\|^2+\|\tilde{c}_{4}\|^2+\tilde{r}^2,
	\end{align*}
	and $k_1$, $k_2$, $\theta_1$ and $\theta_2$ are strictly positive constants.
\end{theorem}
\begin{proof}
	This proof mainly refers to \cite[Chapter 9]{anfinsen2019adaptive}, and makes necessary supplements to the gain approximation error while reducing repetition.
%\end{proof}

	\par A. DeepONet-perturbed target system
	\par We consider the following adaptive backstepping transformation  \eqref{eq:backstepping_1}
    \begin{equation}
    \label{eq:backstepping_NO}
        \begin{aligned}
		w(x,t)=&\hat{u}(x,t),
        \\
		z(x,t)=&\hat{m}(x,t)-\int_{0}^{x}\breve{K}^{u}(x,\xi,t)\hat{u}(\xi,t)d\xi 
        \\
		&- \int_{0}^{x}\breve{K}^{m}(x,\xi,t)\hat{m}(\xi,t)d\xi =: T[\hat{u},\hat{m}](x,t),
	\end{aligned}
    \end{equation}
	where $\breve{K}^{u}$ and $\breve{K}^{m}$ are exact solutions to kernel function \eqref{eq:kernel_1}.
	The transformation is an invertible backstepping transformation, with inverse in the same form
    \begin{equation}
    \label{eq:inverse_2}
        \begin{aligned}
		\hat{u}(x,t)=&w(x,t),
        \\
		\hat{m}(x,t)=& z(x,t)+\int_{0}^{x}\breve{L}^{u}w(\xi,t)d\xi 
        \\ &+\int_{0}^{x}\breve{L}^{m} z(\xi,t)d\xi =:T^{-1}[w,z](x,t) 
	\end{aligned}
    \end{equation}
	where $T^{-1}$ is an operator similar to $T$. From Lemma \ref{lem1-kernel}, $\breve{K}^u,\breve{K}^m$ are continuous, there exist unique continuous inverse kernels $\breve{L}^u,\breve{L}^m$ defined on $\mathcal{T}_1$ and there exists a constant $\bar{L}$ so that $\|\breve{L}^u\|_{\infty} \leq \bar{L}$, $\|\breve{L}^m\|_{\infty} \leq \bar{L}$. We will derive the DeepONet-perturbed target system with exact estimated kernels.
	Because the controller we have chosen is \eqref{eq:U_NO}, where the kernels $\hat{K}^u$ and $\hat{K}^m$ are approximated by NO. 
	This transformation lead to the following target system 
	\begin{align}
		w_t(x,t) =&-\lambda w_x(x,t)+ \hat{c}_{1}w(x,t)+\hat{c}_{1}e_1(x,t)+\hat{c}_{2}z(x,t) \nonumber\\
		&+\int_{0}^{x}\omega(x,\xi,t)w(\xi,t)d\xi \nonumber\\
		&+ \int_{0}^{x}\kappa(x,\xi,t)z(\xi,t)d\xi  \nonumber\\
		&+\hat{c}_{2}e_2(x,t)+\rho e_1(x,t)\|\varpi(t)\|^2, \label{eq:target_NO_1}\\
		z_{t}(x, t)= & \mu z_{x}(x, t)+\hat{c}_{4}z(x,t)- \lambda\breve{K}^{u}(x, 0, t) r(t) e_2(0, t) \nonumber\\
		&- \lambda \breve{K}^{u}(x, 0, t) \tilde{r}(t) z(0, t)+ \lambda\breve{K}^{u}(x, 0, t) e_1(0, t) \nonumber\\
		& -\int_{0}^{x} \breve{K}_{t}^{u}(x, \xi, t) w(\xi, t) d \xi \nonumber\\
		&-\int_{0}^{x} \breve{K}_{t}^{m}(x, \xi, t) T^{-1}[w, z](\xi, t) d \xi \nonumber\\
		& +T\left[\hat{c}_{1}e_1+\hat{c}_{2}e_2, \hat{c}_{3} e_1+\hat{c}_{4}e_2 \right](x, t)\nonumber\\
		&+\rho T[e_1, e_2](x, t)\|\varpi(t)\|^{2}, \\
		w(0, t)= & r(t) z(0, t)+r(t) e_2(0, t)-e_1(0, t), \\
		z(1, t)=& -\int_{0}^{1}\tilde{K}^u(1,\xi,t)w(\xi,t)d\xi\nonumber\\
		&-\int_{0}^{1}\tilde{K}^m(1,\xi,t)T^{-1}[w,z](\xi,t)d\xi:=\Gamma(t)\label{eq:target_NO_4},
	\end{align}
	where 
	\begin{align}
		\omega(x, \xi, t) & =\hat{c}_{2}(x,t) \breve{K}^{u}(x, \xi, t)+\int_{\xi}^{x} \kappa(x, s, t) \breve{K}^{u}(s, \xi, t) d s, \nonumber\\
		\kappa(x, \xi, t) & =\hat{c}_{2}(x,t) \breve{K}^{m}(x, \xi, t)+\int_{\xi}^{x} \kappa(x, s, t) \breve{K}^{m}(s, \xi, t) d s .\nonumber
	\end{align}
	The main difference between the current system \eqref{eq:target_NO_1}-\eqref{eq:target_NO_4} and the system described in \eqref{eq:target_1} is the perturbation $\Gamma(t)$ in the boundary conditions \eqref{eq:target_NO_4}. 
	This difference is due to the controller \eqref{eq:U_NO} using an approximated estimated kernels $\hat{K}^u$ and $\hat{K}^m$ instead of the exact estimated kernels $\breve{K}^u$ and $\breve{K}^m$. The specific derivation process of \eqref{eq:target_NO_4} is as follows
	\begin{align*}
		z(1,t)=&\hat{m}(1,t)-\int_{0}^{1}\breve{K}^{u}(1,\xi,t)\hat{u}(\xi,t)d\xi\\
		&-\int_{0}^{1}\breve{K}^{m}(1,\xi,t)\hat{m}(\xi,t)d\xi\\
		=&U(t)-\int_{0}^{1}\breve{K}^{u}(1,\xi,t)\hat{u}(\xi,t)d\xi\\
		&-\int_{0}^{1}\breve{K}^{m}(1,\xi,t)\hat{m}(\xi,t)d\xi\\
		=&\int_{0}^{1}\hat{K}^u(1,\xi,t)\hat{u}(\xi,t)d\xi +\int_{0}^{1}\hat{K}^m(1,\xi,t)\hat{m}(\xi,t)d\xi\\
		&-\int_{0}^{1}\breve{K}^{u}(1,\xi,t)\hat{u}(\xi,t)d\xi-\int_{0}^{1}\breve{K}^{m}(1,\xi,t)\hat{m}(\xi,t)d\xi\\
		=&-\int_{0}^{1}\tilde{K}^{u}(1,\xi,t)\hat{u}(\xi,t)d\xi-\int_{0}^{1}\tilde{K}^{m}(1,\xi,t)\hat{m}(\xi,t)d\xi. %T^{-1}[w,z]		
	\end{align*}
	By using \eqref{eq:inverse_2}, we obtained \eqref{eq:target_NO_4}. 
	\par In the following part, we introduce the spatial $L^2$ boundedness and regulation of plant and observer states.
	\par We use the following Lyapunov function candidate
	\begin{align}
		V(t):=V_4(t)+aV_5(t),
	\end{align}
	 and
	\begin{align}
		V_4(t)&:=\|w(t)\|_{-\delta}^2=\int_{0}^{1}e^{-\delta x}w^2(x,t)dx,\\
		V_5(t)&:=\|z(t)\|_k^2=\int_{0}^{1}e^{kx}z^2(x,t)dx,
	\end{align} 
	where $a,\delta,k>0$.  We deduce that there exist positive constants \( a_1, a_2 > 0 \) such that
\[
a_1 (\|w(t)\| + \|z(t)\|)^2 \leq V(t) \leq a_2 (\|w(t)\| + \|z(t)\|)^2.
\]
    Before we start the formal calculations of the Lyapunov function, we present the formulas derived from Lemma \ref{lem1-kernel} 
	\begin{align*}
		&\|\breve{K}^u\|_{\infty} \leq \bar{K}, \|\breve{K}^m\|_{\infty} \leq \bar{K}, \\
		&\|\tilde{K}^u\|_{\infty} \leq \epsilon,  \|\tilde{K}^m\|_{\infty} \leq \epsilon ,\\
		%&\left\|\breve{K}^u_{t}(t)\right\| \leq \textcolor{red}{M_1\left\|\hat{c}_{t}(t)\right\|}, \left\|\breve{K}^v_{t}(t)\right\| \leq \textcolor{red}{M_2\left\|\hat{c}_{t}(t)\right\|}, \\
		%&|\Gamma(t)| \leq \varepsilon \bar{\Gamma} \textcolor{red}{(\|u(t)\|+\|m(t)\|)}, \\
		&\|w(t)\| =\|\hat{u}(t)\|,\\
		& \|z(t)\| \leq (1+\bar{K}) \|\hat{m}(t)\| +\bar{K} \|\hat{u}(t)\|,\\
		&\|\hat{m}(t)\| \leq (1+\bar{L})\|z(t)\|+\bar{L} \|w(t)\|,\\
		&|\Gamma(t) |\leq \epsilon \bar{\Gamma}(\|w\|+\|z\|),
	\end{align*}
	where 
	\begin{align*}
		\bar{\Gamma}=1+\bar{L}.
		%M_1=&,\\
		%M_2=& ,\\
		%\bar{\Gamma} =& \mathop{\max}\left\{1+\bar{L}^u,1+\bar{L}^m \right\},\\
	\end{align*}
	%The current work is based on the previous work \cite[Chapter 9]{anfinsen2019adaptive}, with a 
    
    The key difference is $z^2(1,t)=\Gamma(t)^2\neq 0$. This leads to the terms $ \dfrac{1}{a_1}\mu e^k\epsilon^2 \bar{\Gamma}^2 V(t)$ in \eqref{eq:V_5 leq} as follows. There exist positive constants $h_1,h_2,\cdots,h_6$ and nonnegative, integrable function $g_1,g_2,\cdots,g_5$ such that 
	\begin{align}
		\dot{V}_{4}(t) \leq & h_{1} z^{2}(0, t)-\left[\lambda \delta-h_{2}\right] V_{4}(t)+h_{3} V_{5}(t)\nonumber \\
		&+g_{1}(t) V_{4}(t)+g_{2}(t), \nonumber \\
		\dot{V}_{5}(t) \leq &\mu e^k z^2(1,t)-\left[\mu-e^{k} h_{4} \tilde{r}^{2}(t)\right] z^{2}(0, t)+h_{5} V_{4}\nonumber\\
		&-\left[k \mu-h_{6}\right] V_{5}  +g_{3}(t) V_{4}+g_{4}(t) V_{5}+g_{5}(t)\nonumber
        \\
        \leq & \dfrac{1}{a_1}\mu e^k\epsilon^2 \bar{\Gamma}^2 V(t)-\left[\mu-e^{k} h_{4} \tilde{r}^{2}(t)\right] z^{2}(0, t)+h_{5} V_{4}\nonumber\\
		&-\left[k \mu-h_{6}\right] V_{5}  +g_{3}(t) V_{4}+g_{4}(t) V_{5}+g_{5}(t), \label{eq:V_5 leq}
	\end{align}
where 
\begin{align*}
	g_{1}(t)=&e^{2 \delta} \rho^{2}\|e_1(t)\|^{2}\|\varpi(t)\|^{2},\\
	g_2(t)=& \left(\bar{c}_{1}^{2}+4 e^{-\delta}\right)\|e_1(t)\|^{2}+\left(\bar{c}_{2}^{2}+4 e^{-\delta}\right)\|e_2(t)\|^{2} \nonumber\\
	& +3 \lambda e_1^{2}(0, t)+3 \lambda \bar{r}^{2} e_2^{2}(0, t),\\
	g_{3}(t)= & e^{\delta+k}\left\|\hat{K}_{t}^{u}(t)\right\|^{2}+2 e^{\delta+k}\left\|\hat{K}_{t}^{m}(t)\right\|^{2} A_{3}^{2}, \\
	g_{4}(t)= & 2 \rho^{2} e^{\delta+2 k}\left(A_{1}^{2}\|e_1(t)\|^{2}+A_{2}^{2}\|e_2(t)\|^{2}\right)\|\varpi(t)\|^{2} \nonumber\\
	& +2 e^{k}\left\|\hat{K}_{t}^{m}(t)\right\|^{2} A_{4}^{2},\\
	g_5(t)=&\lambda^{2} \bar{K}^{2} \bar{r}^{2} e^{k} e_2^{2}(0, t)+\lambda^{2} \bar{K}^{2} e^{k} e_1^{2}(0, t)\nonumber\\
	&+2 e^{k}\Big(A_{1}^{2} \bar{c}_{1}^{2}\|e_1(t)\|^{2}+A_{1}^{2} \bar{c}_{2}^{2}\|e_2(t)\|^{2}+A_{2}^{2} \bar{c}_{3}^{2}\|e_1(t)\|^{2}  \nonumber\\
	& +A_{2}^{2} \bar{c}_{4}^{2}\|e_2(t)\|^{2}\Big)+4 e^{-\delta}\Big(|e_1(t)\|^{2}+\|e_2(t)\|^{2}\Big),
\end{align*}
 where $A_i>0$, $i=1\cdots4$, $\delta \geq 1$. Thus, we obtain the following upper bound calculation

	\begin{equation}\label{eq:V6_dot1_1}
		%\dot{V}_6(t) \leq-\left[c-{2\mu e^k\epsilon^2 \bar{\Gamma}^2} \right] V_6(t)+l_6(t) V_6(t)+l_{10}(t)
		\dot{V}(t) \leq-\left[d- \dfrac{1}{a_1}\mu e^k\epsilon^2 \bar{\Gamma}^2  \right] V(t)+g_6(t) V(t)+g_{7}(t),
	\end{equation}
	for positive constant $d$ and the nonnegative, integrable functions $g_{6}(t)$ and $g_7(t)$
	\begin{align}
		%l_{10}(t) =& l_9(t) + b\sigma^2(t)m^2(0, t),\\
		g_{6}(t) =&\mathop{\max}
		\left\{g_1(t)+ag_3(t),g_4(t)\right\}, \label{eq:l6}\\
		g_{7}(t) =& \dfrac{2\tilde{r}^2z(0,t)^2}{1+m(0,t)^2}+8b\tilde{r}^2e_2(0,t)^2+g_2(t)+ag_5(t)\nonumber\\
		&+b\dfrac{\tilde{r}^2z(0,t)^2}{1+m(0,t)^2}z(0,t)^2,\label{eq:l10}
	\end{align}
	where $a,b$ are positive constants.
	We introduce the positive constant
	\begin{equation*}
		\epsilon _0:=\dfrac{\sqrt{a_1(2d-1)}}{\sqrt{2\mu e^k \bar{\Gamma}^2}}.
	\end{equation*} 
	Thus, if we choose $\epsilon \in (0,\epsilon_0)$ we have $d- \dfrac{1}{a_1}\mu e^k\epsilon^2 \bar{\Gamma}^2 >{1}/{2}>0$.
	It then follows from Lemma 12 in \cite{anfinsen2017model} that
	%Lemma B. 3 in \cite{anfinsen2019adaptive} that
	$$
	V \in L^1 \cap L^{\infty},
	$$
	and hence
	$$
	\|w\|,\|z\| \in L^2 \cap L^{\infty} .
	$$
	Due to the invertibility of the backstepping transformation                                  
	$$
	\|\hat{u}\|,\|\hat{m}\| \in L^2 \cap L^{\infty} .
	$$
    From Lemma \ref{lemma:Identifier_bound}, $\|e_1\|,\|e_2\|\in L^2 \cap L^{\infty}$, it follows that 
    $$
	\|{u}\|,\|{m}\| \in L^2 \cap L^{\infty} .$$

	B. Pointwise-in-space boundedness and regulation
	\par The paper \cite{vazquez2011backstepping} proved that the system \eqref{eq:uv_def} is equivalent to the following system through an invertible backstepping transformation.
	\begin{align}\label{alpha_1}
		\psi_t(x, t)&= -\lambda \psi_x(x, t)+h_1(x) \zeta(0, t), \\
		\zeta_t(x, t)&=\mu \zeta_x(x, t), \\
		\psi(0, t)&= r \zeta(0, t), \\
		\zeta(1, t)%&= \int_0^1\left[\hat{K}^{u m}(1, \xi, t) \hat{u}(\xi, t)+\hat{K}^{m m}(1, \xi, t) \hat{m}(\xi, t)\right] d \xi \\
		%&-\int_0^1\left[K^{u m}(1, \xi) u(\xi, t)+K^{m m}(1, \xi) m(\xi, t)\right] d \xi \\
		&=U(t) -\int_{0}^{1}(G_1(\xi)u(\xi) -G_2(\xi)m(\xi))d\xi, \label{alpha_4}
	\end{align}
	for some bounded functions $h_1,G_1,G_2$ of the unknown parameters. Equation \eqref{alpha_1}-\eqref{alpha_4} can be explicitly be solved for $t>\lambda^{-1}+\mu^{-1}$ to yield 
	\begin{align}\label{eq:alpha_solution}
		\psi&(x, t)=r \zeta\left(1, t-\mu^{-1}-\lambda^{-1} x\right) \nonumber \\
		&+\lambda^{-1} \int_{0}^{x} h_1(\tau) \zeta\left(1, t-\mu^{-1}-\lambda^{-1}(x-\tau)\right) d \tau, \\
		\zeta&(x, t)=\zeta\left(1, t-\mu^{-1}(1-x)\right) . \label{eq:alpha_solution_2}
	\end{align}
	From \eqref{alpha_4}, the control law $U(t)$ and $\|u\|,\|m\|,\|\hat{u}\|,\|\hat{m}\| \in$ $L^2 \cap L^{\infty}$, it follows that $\zeta(1, \cdot) \in$ $L^2 \cap L^{\infty}$. Since $\zeta$ and $\psi$ are simple, cascaded transport equations, this implies
	\begin{align*}
		\|\psi\|_{\infty},\|\zeta\|_{\infty} \in L^2 \cap L^{\infty}, \quad\|\psi\|_{\infty},\|\zeta\|_{\infty} \rightarrow 0.
	\end{align*}
	With the invertibility of the transformation, then yields
	\begin{equation*}
		\|u\|_{\infty},\|m\|_{\infty} \in L^2 \cap L^{\infty}, \quad\|u\|_{\infty},\|m\|_{\infty} \rightarrow 0.
	\end{equation*}

	From the structure of the identifier \eqref{eq:uv_def}, we will also have  $\hat{u}(x, \cdot), \hat{m}(x, \cdot) \in L^{2} \cap L^{\infty}$ , and hence
	\begin{align*}
		\|\hat{u}\|_{\infty}, \|\hat{m}\|_{\infty} \in L^{2} \cap L^{\infty}, 
		\quad \|\hat{u}\|_{\infty},\|\hat{m}\|_{\infty} \mapsto 0.
	\end{align*}

	C. Global stability
	\par Here, we will prove the global stability of the system, specifically by proving \eqref{eq:S(t)}, and thus we introduce the following function
	\begin{align*}
		S(t) :=& \|u\|^2+\|m\|^2+\|\hat{u}\|^2+\|\hat{m}\|^2+\|\tilde{c}_{1}\|^2+\|\tilde{c}_{2}\|^2\nonumber\\
		&+\|\tilde{c}_{3}\|^2+\|\tilde{c}_{4}\|^2+\tilde{r}^2.
	\end{align*}
	The goal of the proof is to demonstrate the existence of a function $\theta $ such that the following inequality holds. 
	%The goal is to prove the existence of a function $\theta \in \mathcal{K}_{\infty}$ such that
	\begin{equation*}
		S(t) \leq \theta(S(0)), t\geq 0.
	\end{equation*}
	%For that we reuse the Lyapunov function
	We will reuse the Lyapunov function from Appendix A to show that the system's state remains stable over time.
	\begin{align*}
		V_1(t) =& V_2(t)+\gamma_1^{-1}\|\tilde{c}_{1}\|^2+\gamma_2^{-1}\|\tilde{c}_{2}\|^2+\gamma_3^{-1}\|\tilde{c}_{3}\|^2 \nonumber\\
		&+\gamma_4^{-1}\|\tilde{c}_{4}\|^2+\dfrac{\lambda}{2\gamma_5}\tilde{r}^2(t),
	\end{align*}
	where 
	\begin{equation*}
		V_2(t) = \int_{0}^{1}e^{-\gamma x}e_1^2(x,t)dx + \int_{0}^{1}e^{\gamma x}e_2^2(x,t)dx,
	\end{equation*}
	leads to the following upper bound:
	\begin{align}
		\dot{V}_{1}(t) \leq & -\lambda e^{-\gamma} e_1^{2}(1, t)-\lambda e_1^{2}(0, t) m^{2}(0, t)-\lambda \gamma e^{-\gamma}\|e_1(t)\|^{2} \nonumber\\
		& -2 \rho e^{-\gamma}\|e_1(t)\|^{2}\|\varpi(t)\|^{2}-\mu e_2^{2}(0, t) \nonumber\\
		& -\mu \gamma\|e_2(t)\|^{2}-2 \rho e^{\gamma}\|e_2(t)\|^{2}\|\varpi(t)\|^{2}, \label{eq:dotV_1}
	\end{align}
	which shows that ${V}_1(t)$ is non-increasing and hence bounded. Thus implies that the ${V}_1(t)<V_1(0)$ and limit $\lim_{t\to \infty}V_1(t)=V_{1,\infty}$ exists. By integrating \eqref{eq:dotV_1} from zero to infinity, we obtain the following upper bound:
	\begin{align*}
		&\lambda e^{-\gamma} \int_{0}^{\infty}e_1^{2}(1, \tau)d\tau+\lambda \int_{0}^{\infty}e_1^{2}(0, \tau) m^{2}(0, \tau)d\tau \nonumber\\
		&+\lambda \gamma e^{-\gamma}\int_{0}^{\infty}\|e_1(\tau)\|^{2}d\tau \nonumber\\
		& +2 \rho e^{-\gamma}\int_{0}^{\infty}\|e_1(\tau)\|^{2}\|\varpi(\tau)\|^{2}d\tau+\mu \int_{0}^{\infty}e_2^{2}(0, \tau)d\tau \nonumber\\
		& +\mu \gamma \int_{0}^{\infty}\|e_2(\tau)\|^{2}d\tau+2 \rho e^{\gamma}\int_{0}^{\infty}\|e_2(\tau)\|^{2}\|\varpi(\tau)\|^{2}d\tau \nonumber\\
		&\leq V_1(0).
	\end{align*}
	
	From \eqref{eq:l6} and \eqref{eq:l10}, it can be concluded that there are constants $\theta_1 >0 $ and  $\theta_2>1$ such that
	\begin{align}\label{eq:l6_1}
		\|g_6\|_1 \leq& \theta_1 V_1(0), \\
		\|g_{7}\|_1 \leq& \theta_2 V_1(0). \label{eq:l10_1}
	\end{align}
	Recalling \eqref{eq:V6_dot1_1}, we have that %$(V=V_6)$
	\begin{align*}
		\dot{V}(t) \leq -\frac{1}{2}V(t) +g_6(t)V(t) +g_{7}(t).
	\end{align*}
	We also have from Lemma B.6 in \cite{krstic1995nonlinear} that 
	\begin{equation}\label{eq:V leq}
		V(t)\leq (e^{-\frac{1}{2}t}V(0)+\|g_{7}\|_1)e^{\|g_6\|_1}.
	\end{equation}
	\par We then introduce the function
	\begin{equation*}
		V_6(t):=V_1(t)+V(t).
	\end{equation*}
	Noticing that
	\begin{equation}\label{eq:V1 leq}
		V_1(t)\leq 	V_1(0)\leq \theta_2 V_1(0)e^{\theta_1 V_1(0)},		
	\end{equation}
	we achieve from \eqref{eq:V leq}, \eqref{eq:V1 leq}, \eqref{eq:l6_1} and \eqref{eq:l10_1} the following
	\begin{align*}
		V_6(t)= & V(t)+V_1(t)\nonumber \\
		\leq &  (e^{-\frac{1}{2}t}V(0)+\|g_{7}\|_1)e^{\|g_6\|_1} +\theta_2 V_1(0)e^{\theta_1 V_1(0)}	\nonumber	\\
		\leq & (\theta_2 V(0)+\theta_2V_1(0))e^{\|g_6\|_1} +\theta_2 V_1(0)e^{\theta_1 V_1(0)} \nonumber\\
		\leq &2\theta_2 V_6(0) e^{\theta_1 V_6(0)}.
	\end{align*}
	This Lyapunov functional can be represented by an equivalent norm, and the bounds of this equivalent norm are determined by two positive constants $k_1 > 0$ and $k_2 > 0$.
	\begin{equation*}
		k_1 S(t)\leq V_6(t) \leq k_2 S(t).
	\end{equation*}
	So we have
	\begin{equation*}
		S(t)\leq 2\dfrac{k_2}{k_1}\theta_2 S(0) e^{\theta_1 k_2S(0)}.
	\end{equation*}
\end{proof}

\section{Simulations}\label{sec_Simulation}
This section will present and analyze the performance of the proposed NO-based adaptive controllers for two PDE models: (i) a general 2$\times$2 hyperbolic system \eqref{eq:uv_def} (ii) the ARZ PDE system. Through these examples, we will demonstrate the effectiveness of the NO-based adaptive control design.
\subsection{Simulation of the Coupled 2$\times$2 Hyperbolic System}

\textit{A. Simulation configuration} 

The coefficients are defined as $c_1 (x)=\cos (\sigma_1 \cos ^ {-1} (x)) $, $c_2 (x)=\cos (\sigma_2 \cos ^ {-1} (x))$, $c_3 (x)=\sin(1-\sigma_3 x)+1 $ and $c_4 (x)=\cos (\sigma_4 x)$ with the shape parameters $\sigma_1$, $\sigma_2 $, $\sigma_3 $ and  $\sigma_4$. Although this paper uses specific Chebyshev polynomial forms, sine functions, and cosine functions, our framework is applicable to any compact set of continuous functions. \textcolor{black}{We use the first order finite difference scheme to solve PDEs, where the time step $dt=0.005s$, the spatial step $dx=0.05m $}, the total time $T=10s$, and the length $L=1m$. The initial conditions are $u_0=\sin(2\pi x)$, $m_0=x$. 

\par\textit{ B. Dataset generation and NO training }%for Kernel Function Approximation} 

We choose 10 sets of  $(c_1,c_2,c_3,c_4,r)$ randomly sampled with $\sigma_1 \sim U (3.5,4.5) $, $\sigma_2 \sim U (0.8,1) $, $\sigma_3 \sim U(20,21)$, $\sigma_4 \sim U(10,11)$ and $r \sim U (2,5) $, where $U(a,b)$ denotes the uniform distribution over the interval $[a,b]$. We simulate trajectories using adaptive control methods and calculate the corresponding kernel functions using numerical solvers.
\textcolor{black}{In practice, during DeepONet design, we choose a sufficiently expressive architecture to ensure accurate approximation of the kernel functions as shown in Figure \ref{structure}.}
Each trajectory was sampled at $1000$ time points, resulting in a dataset of $10000$ sets of $(\hat{c}_1,\hat{c}_2,\hat{c}_3,\hat{c}_4,\hat{r}, \breve{K}^u,\breve{K}^m)$ for training. 
We trained the model on an Nvidia RTX 4060 Ti GPU. After 600 epochs of training, the $L^2$ error of neural operator $\mathcal{\hat{K}}$ reached $1.2 \times 10 ^ {-3} $, and the test error was $1.1 \times 10^{-3}$, as shown in Figure \ref{loss}. 
\par \textit{C. Computation time comparison} 

Table \ref {tab1} provides a comparison of the computation time of solving kernels at each time step using the numerical solver and the trained DeepONet model. We can see that as the sampling accuracy improves, the acceleration obtained by the NO becomes substantial. 
We computed the average absolute error $\int_{\xi}^{1}\int_{0}^{x}(|\breve{K}^u-\hat{K}^u|+|\breve{K}^m-\hat{K}^m|)d\xi dx$ between numerical solutions and NO solutions with different step sizes. Although the error slightly increases with the decrease of step size, they are quite small at all step sizes. Because adaptive control requires calculating control gain at every step of updating parameter estimation, quickly solving the kernel function can help improve the performance of adaptive control.

\par \textit{D. Simulation results }

We test the performance of the closed-loop system stability with test values $(\sigma_1 =4,\sigma_2 =0.9,\sigma_3 =20.1,\sigma_4 =10.1,r=4)$ unseen during training. Figure \ref{kernel} shows the kernels $\breve{K}^u, \breve{K}^m$ calculated by the numerical solver, the kernels $ \hat{K}^u, \hat{K}^m$ learned by DeepONet, and the error between them. 
In Figure \ref{uv}, we demonstrate closed-loop stability with the NO approximated kernel function for the control feedback law. 
Figure \ref{kernel} and Figure \ref{uv} confirm that the kernels $\hat{K}^u, \hat{K}^m$ approximated by NO can effectively simulate the backstepping kernels $\breve{K}^u, \breve{K}^m$ while maintaining the stability of the system. All estimated parameters $\hat{c}_i$ and $\hat{r}$ are shown in Figure \ref{coefficient ci}. 
We emphasize that although in adaptive control the system parameters $\hat{c}_1,\hat{c}_2,\hat{c}_3,\hat{c}_4$ and $\hat{r}$ may not precisely converge to their true values, this does not affect the control performance. This phenomenon is not a problem but rather a characteristic of adaptive control. The goal of adaptive control is not perfect system identification, but rather the estimation of parameters that ensure system stability.

\par \textit{E. Comparative experiment with RL}

We will evaluate the performance of NO-based adaptive control method and RL method for stabilization results under different initial conditions. \textcolor{black}{In this work, we implement the Proximal Policy Optimization (PPO) algorithm. The PDE state is discretized and used as the observation input to a neural network policy. The output of the policy network determines the boundary control action at each time step. The PPO algorithm is trained to minimize a cumulative cost function, which achieves regulation of the traffic states to a spatially uniform density and velocity. We use the standard clipped surrogate objective for policy updates.} We choose the initial condition of state $u$ is a sine function 
\begin{equation}
\label{initial_w0}
u_0=\sin(\omega_0 \pi x),
\end{equation}
where $\omega_0$ is the frequency of a sine wave. To evaluate the performance of these two methods, we train DeepONet and RL at the same frequency $\omega_0=2 $, ensuring all other parameters remained consistent with those in Figure \ref{kernel}. 
In the testing phase, we will use sine initial conditions of different frequencies  $\omega_0=2,10 $ to verify the model stability of NO-based adaptive control and RL. Figure \ref{Comparison with RL} shows the stabilization results of the RL and NO control under different initial conditions. 
The comparative experiments highlight a significant advantage of the NO-based adaptive control method, which consistently demonstrates robustness across different initial conditions. Specifically, the NO-based adaptive control method maintains system stability without requiring retraining even when the initial conditions are changed. This characteristic underscores its adaptability in dynamic environments.
In contrast, the RL method shows a significant dependency on initial conditions. Although it performs well under specific conditions encountered during training, it is unstable when faced with unforeseen initial conditions($\omega_0=10$).In real-world scenarios where initial conditions are often variable and unpredictable, DeepONet ensures stability and adaptability without the need for retraining. In summary, this demonstrates DeepONet’s potential for more reliable applications in adaptive control systems, where maintaining performance across diverse conditions is crucial.

\begin{table}[t]

\centering
\resizebox{\columnwidth}{!}{%
\begin{tabular}{lcccc}
\hline
\textbf{\begin{tabular}[c]{@{}l@{}}Spatial Step  Size \end{tabular}} & \multicolumn{1}{l}{\textbf{\begin{tabular}[c]{@{}l@{}}Numerical solver (s)  \end{tabular}}} & \multicolumn{1}{l}{\textbf{\begin{tabular}[c]{@{}l@{}}NO(s)  \end{tabular}}} & \multicolumn{1}{l}{\textbf{Speedup}  } & \multicolumn{1}{l}{\textbf{Error}  } \\ \hline
dx=$0.01$                                                                     & $8.221\times 10^{-3}$                                                                                                             & $4.38\times 10^{-3}$                                                                                                                &  $2\times$    &  $0.024$    \\
dx=$0.05$                                                                     & $2.432\times 10^{-2}$                                                                                                             & $4.42\times 10^{-3}$                                                                                                                &  $58\times$       &  $0.031$    \\
dx=$0.001$                                                                    & $8.701\times 10^{-1}$                                                                                                               & $4.513\times 10^{-3}$                                                                                                                &  $192\times$        &  $0.037$    \\

dx=$0.0005$                                                                    & $3.191$                                                                                                                & $4.631\times 10^{-3}$                                                                                                                & $689\times$        &  $0.045$    \\
\hline
\end{tabular}%

} 
\caption{Comparison of computation time of kernels $\breve{K}^u$ and $\breve{K}^m$. }
\label{tab1}
\end{table}

\definecolor{lighter-green}{rgb}{0, 0.6, 0.00392156862}

\begin{figure}[ht]
\centering
\includegraphics[width=7cm]{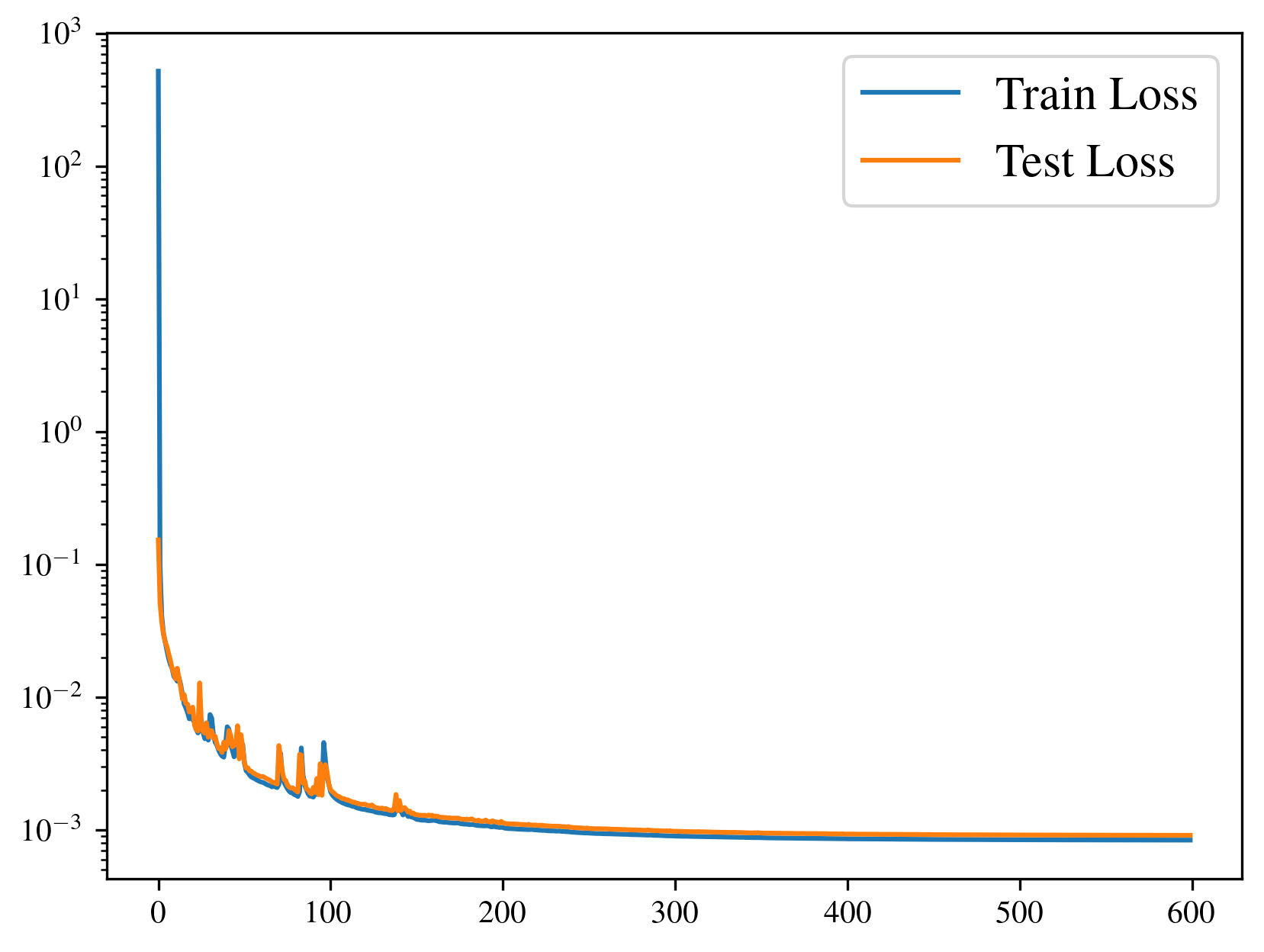} %14cm
\caption{The train and test loss for neural operator $\mathcal{\hat{K}}$}
\label{loss}
\end{figure}

%kernel Ku Kv
\begin{figure*}[ht]
\centering
\includegraphics[width=15cm]{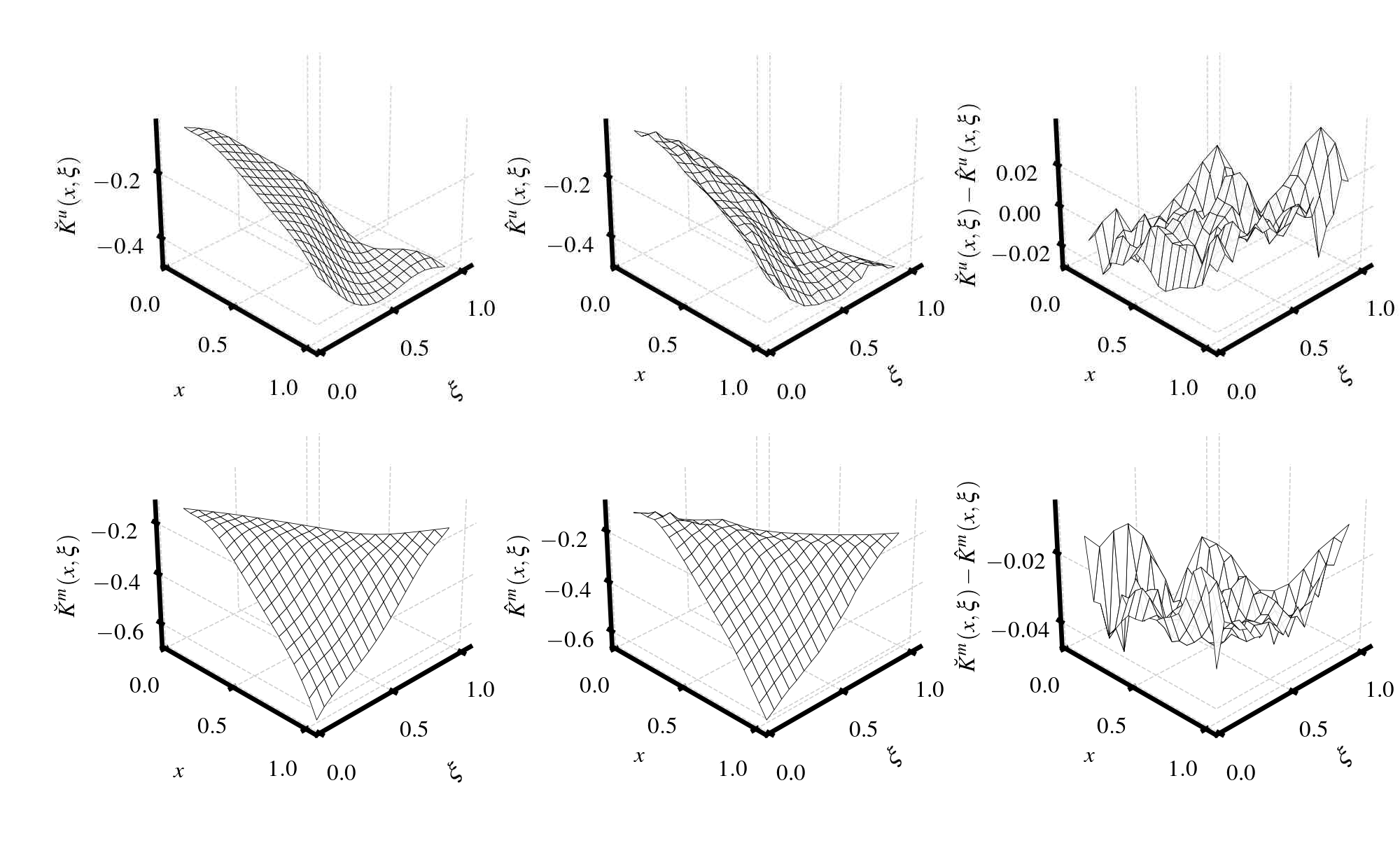}
\caption{The first row shows the exact estimated kernels $\breve{K}^u$ and $\breve{K}^m$. The second row shows the NO estimated kernels $\hat{K}^u$ and $\hat{K}^m$. The last row shows the kernel errors $\breve{K}^u-\hat{K}^u$ and $\breve{K}^m-\hat{K}^m$. All kernels are plotted at the final time $T$.}
\label{kernel}
\end{figure*}

%--------------------state-------------------------------
\begin{figure*}
\centering
\includegraphics[width=18cm]{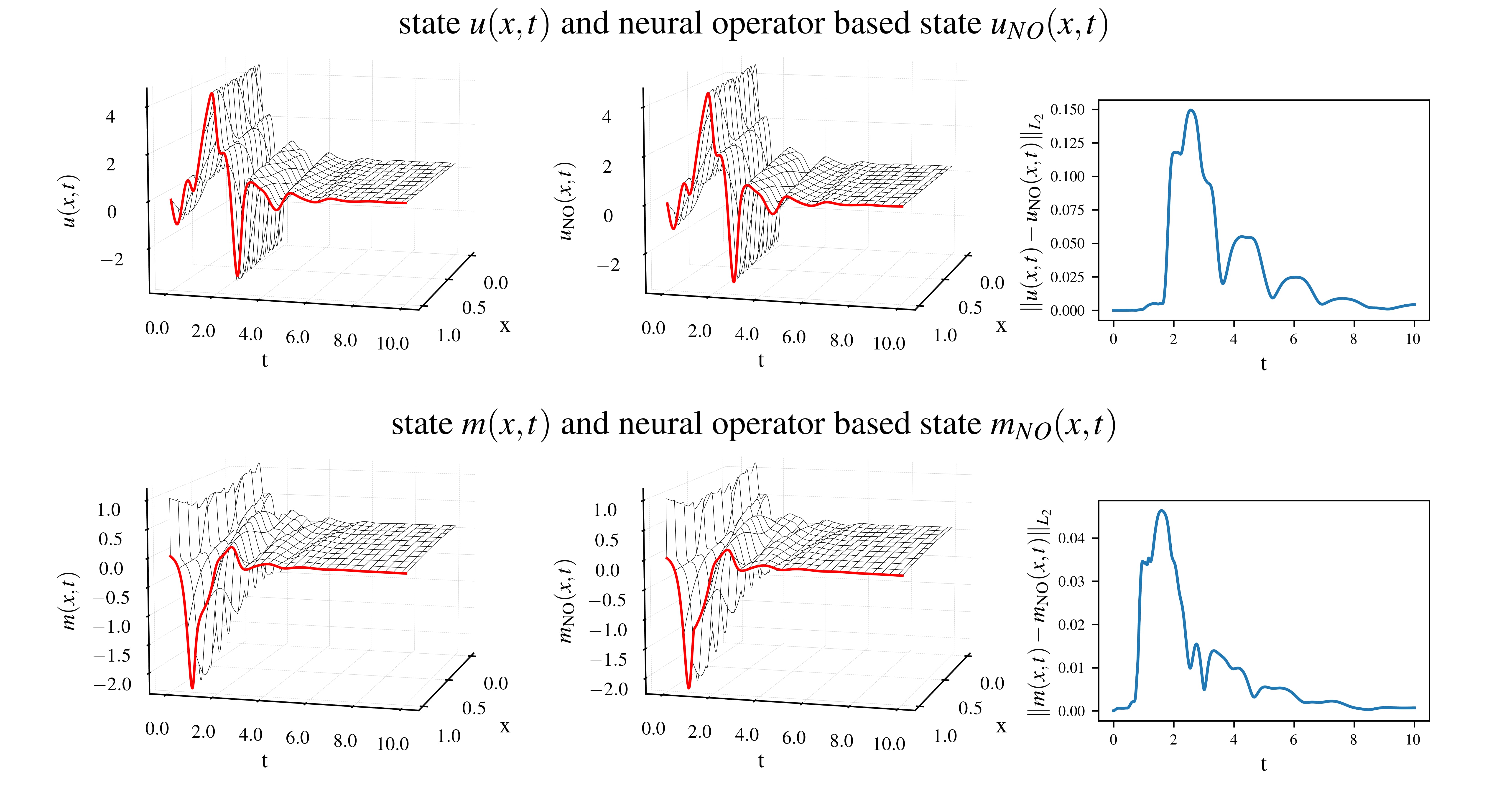}
\caption{Simulation of the close-loop system with feedback controller \eqref{U_exact} and \eqref{eq:U_NO}. The left columns of the first and second rows show close-loop system states $u(x,t)$, $m(x,t)$ with the analyzed kernels $\breve{K}^u$ and $\breve{K}^m$. The middle columns of the first and second rows show close-loop system states $u_{NO}(x,t)$, $m_{NO}(x,t)$ with the approximated kernels $\hat{K}^u$ and $\hat{K}^m$. The right columns of the first and second rows show the errors between  $u(x,t)$ and  $u_{NO}(x,t)$, and between $m(x,t)$ and  $m_{NO}(x,t)$, respectively. 
}
\label{uv}
\end{figure*} 
%-----------coefficient c1 c2--------------------------

\begin{figure*}[ht]
\centering
% 第一行：3个子图
\subfigure[$\hat{c}_1$]{
\includegraphics[width=0.3\textwidth]{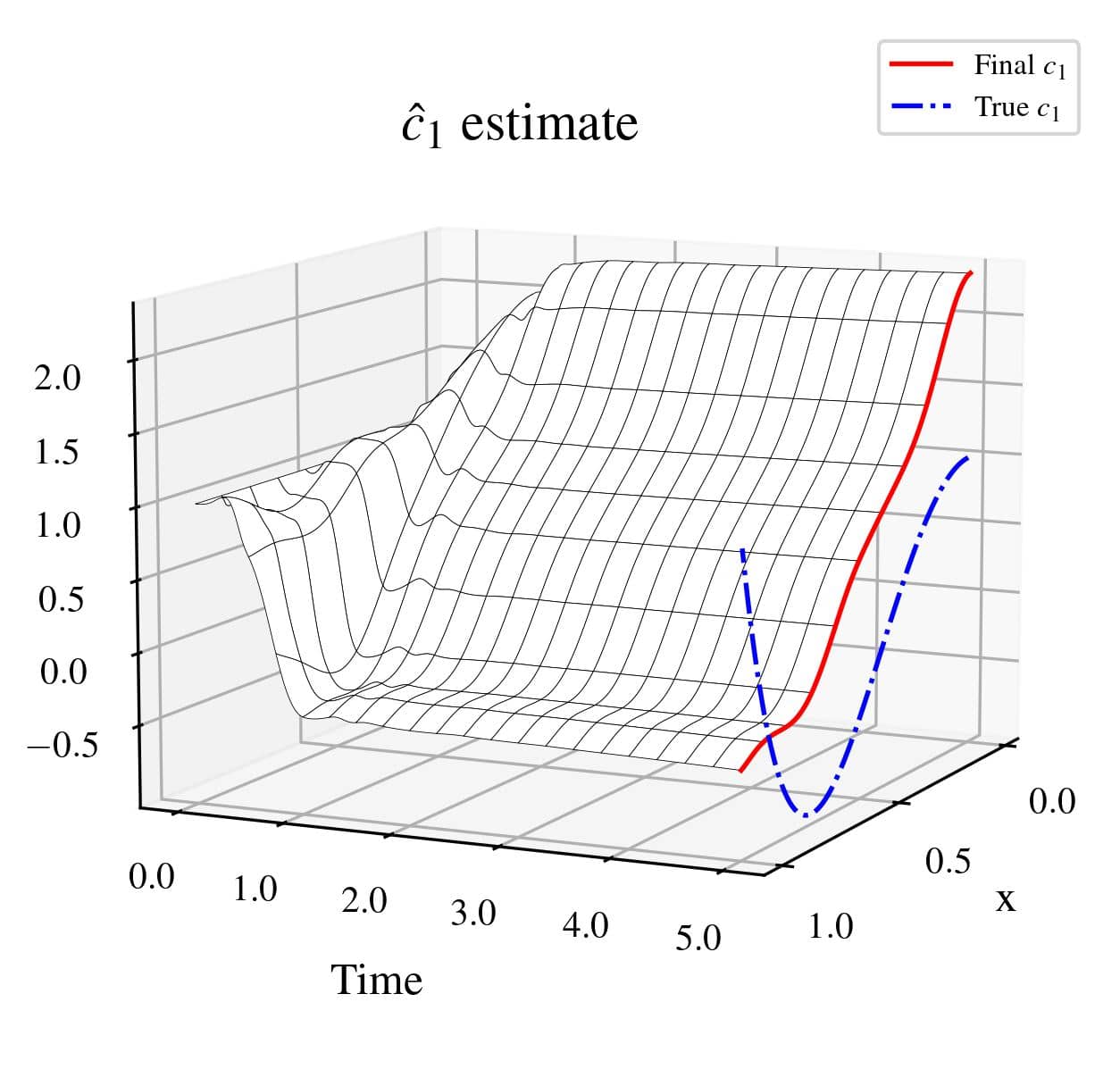} % 替换为你的图片路径
}
\subfigure[$\hat{c}_2$]{
\includegraphics[width=0.3\textwidth]{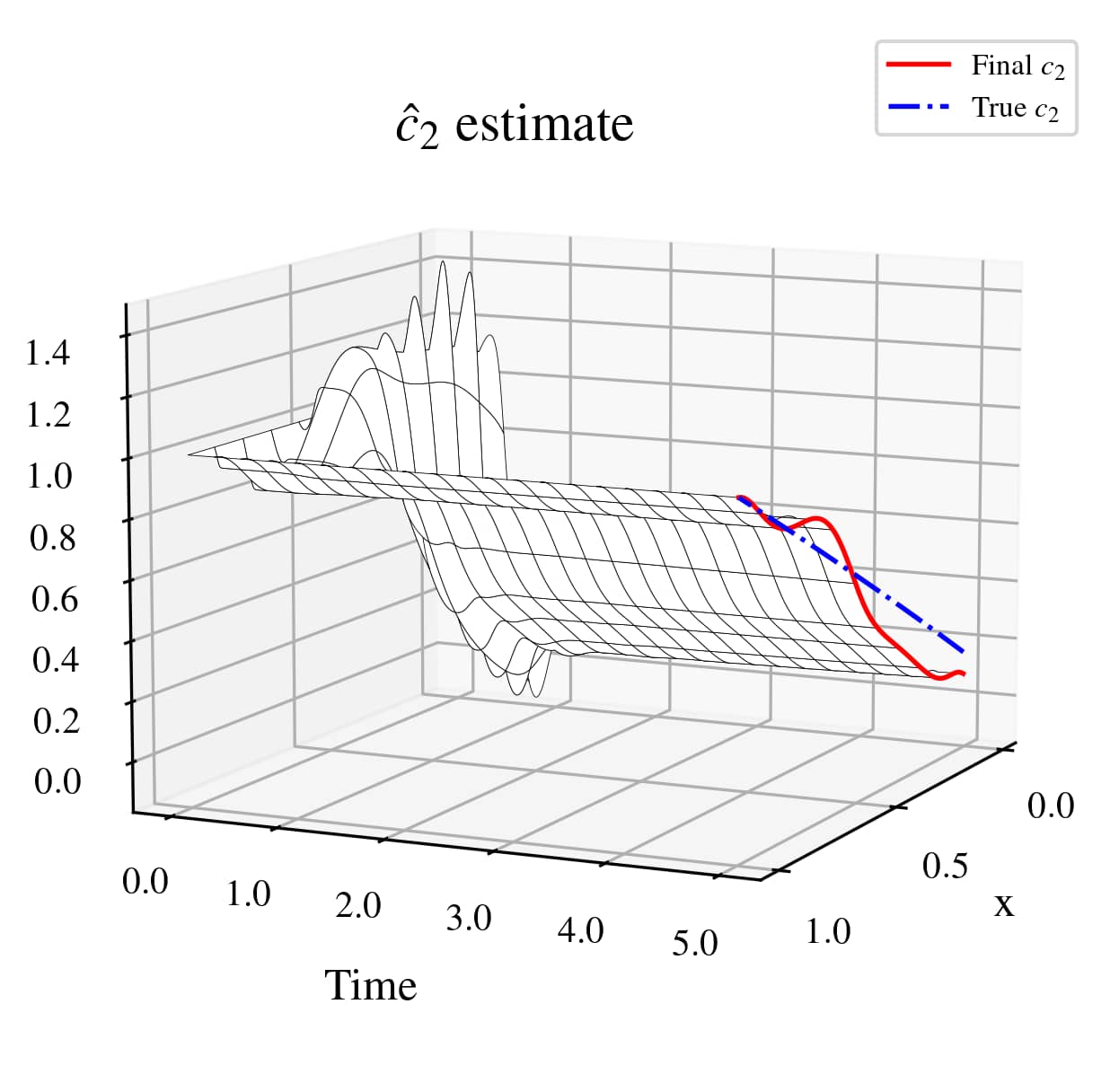}
}
\subfigure[$\hat{c}_3$]{
\includegraphics[width=0.3\textwidth]{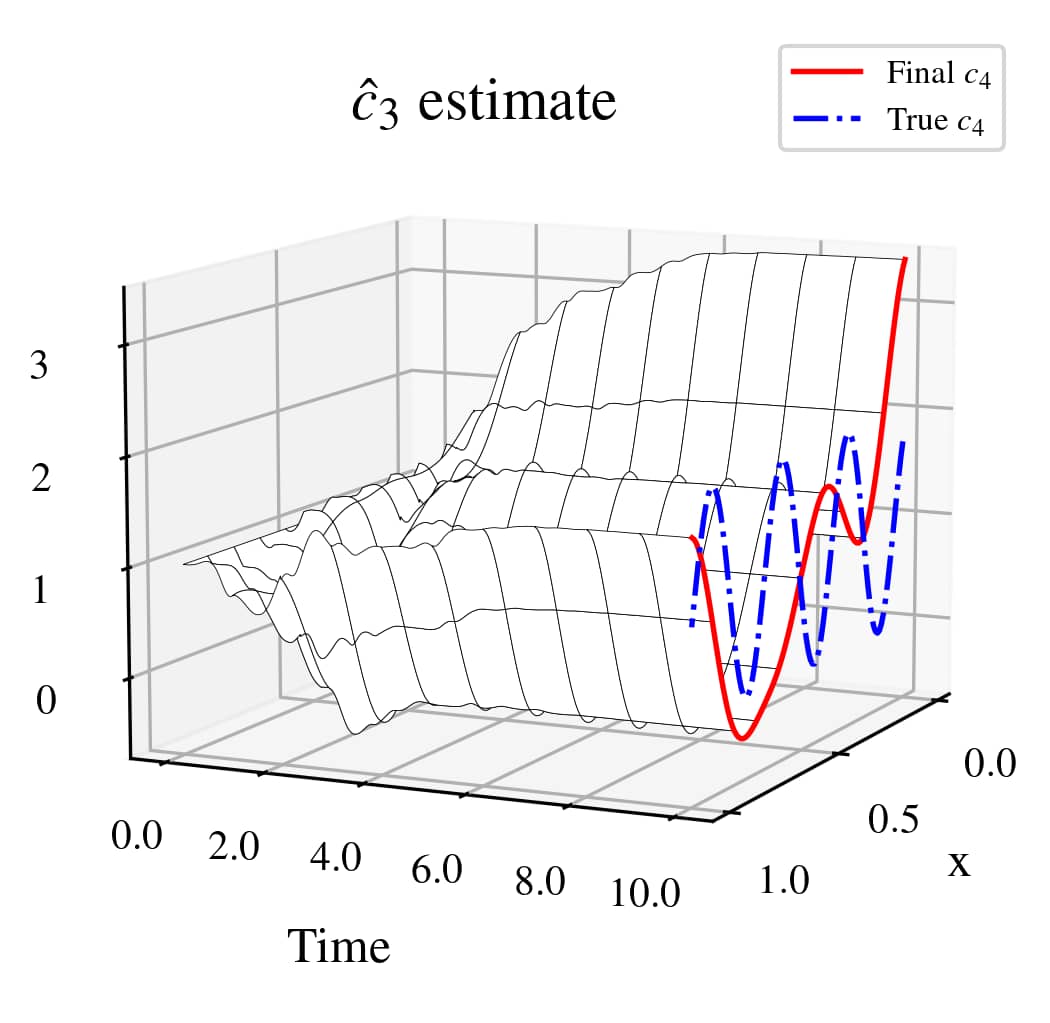}
}
% 第二行：2个居中子图
\subfigure[$\hat{c}_4$]{
\includegraphics[width=0.3\textwidth]{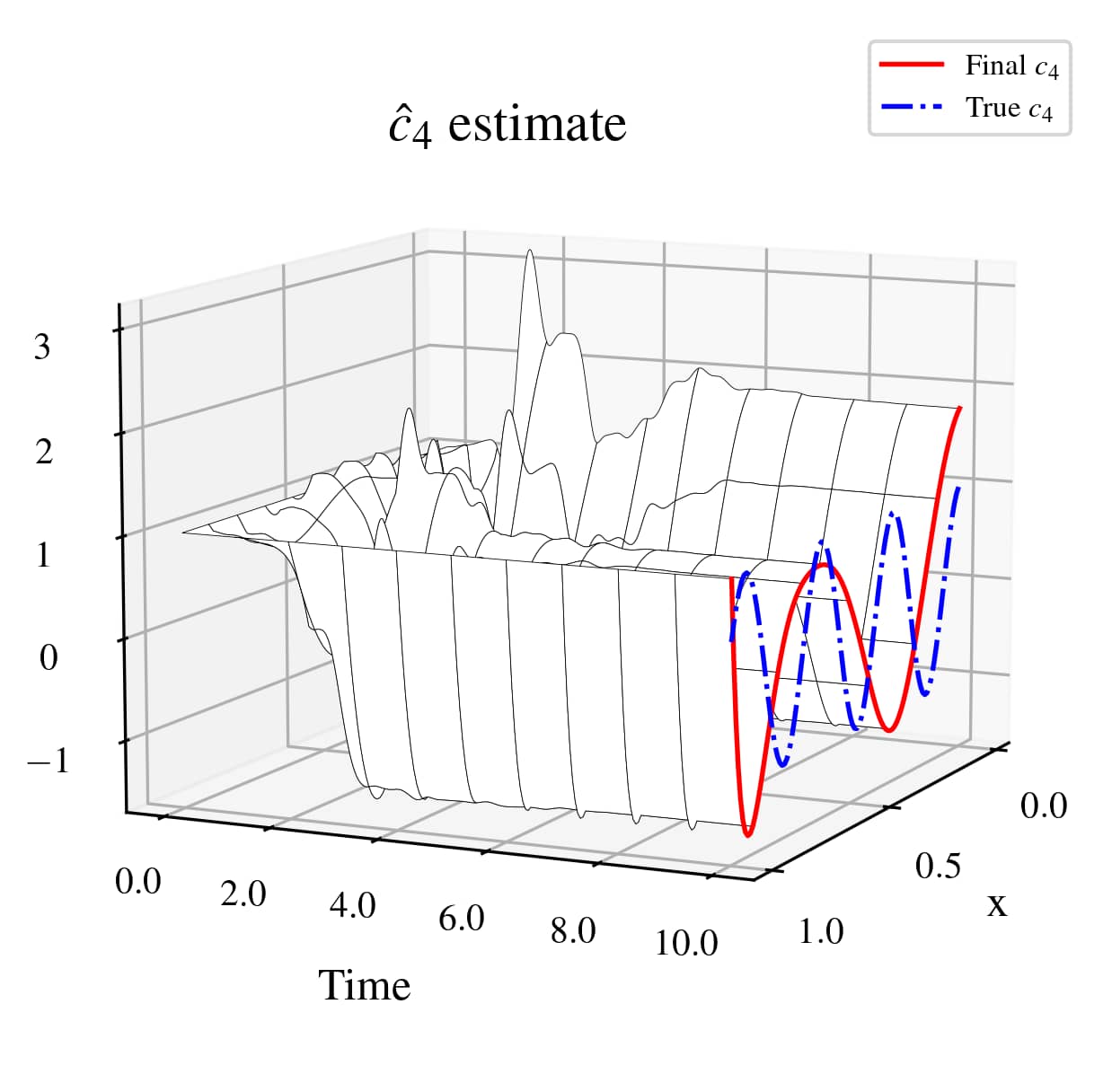}
}
\subfigure[$\hat{r}$]{
\includegraphics[width=0.28\textwidth]{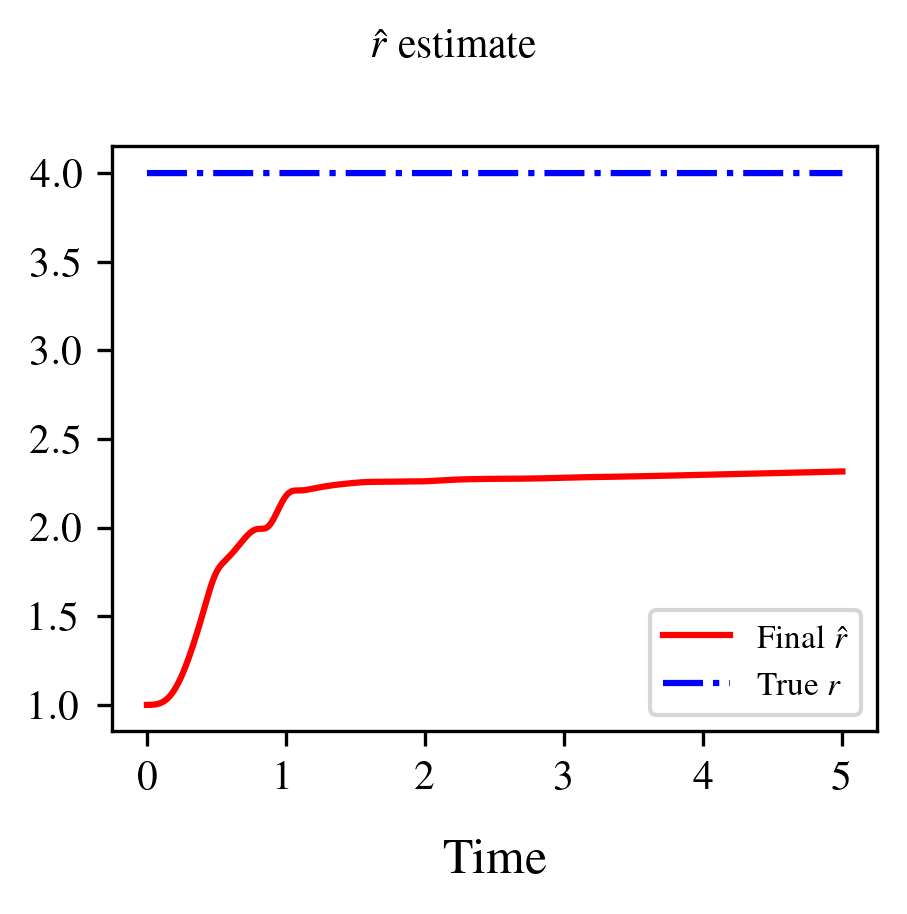}
}
\caption{$\hat{c}_1$, $\hat{c}_2$, $\hat{c}_3$, $\hat{c}_4$ and $\hat{r}$ estimates using NO approximated kernels. The blue dashed lines represent the true parameters, and the red solid lines represent the estimated parameters.}
\label{coefficient ci}
\end{figure*}

%%%%%%%%%%%%% --------对比RL--------------%%%%%%%%%%%%%%%%%%%%%%%%%%%%%
\begin{figure*}[ht]
\centering
% 第一行：3个子图
\subfigure[$u_{NO}(x, t)$ and $m_{NO}(x, t)$ for $\omega_0=2$]{
\includegraphics[width=0.48\textwidth]{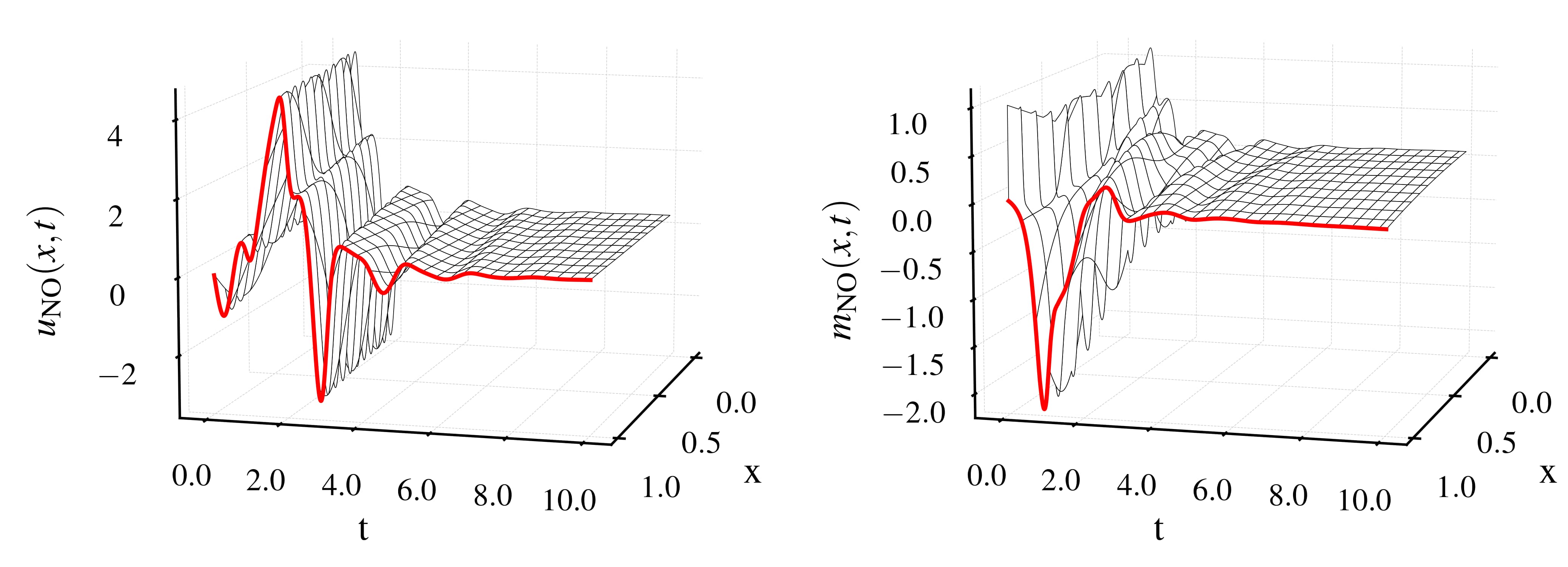} 
}
\subfigure[$u_{NO}(x, t)$ and $m_{NO}(x, t)$ for $\omega_0=10$]{
\includegraphics[width=0.48\textwidth]{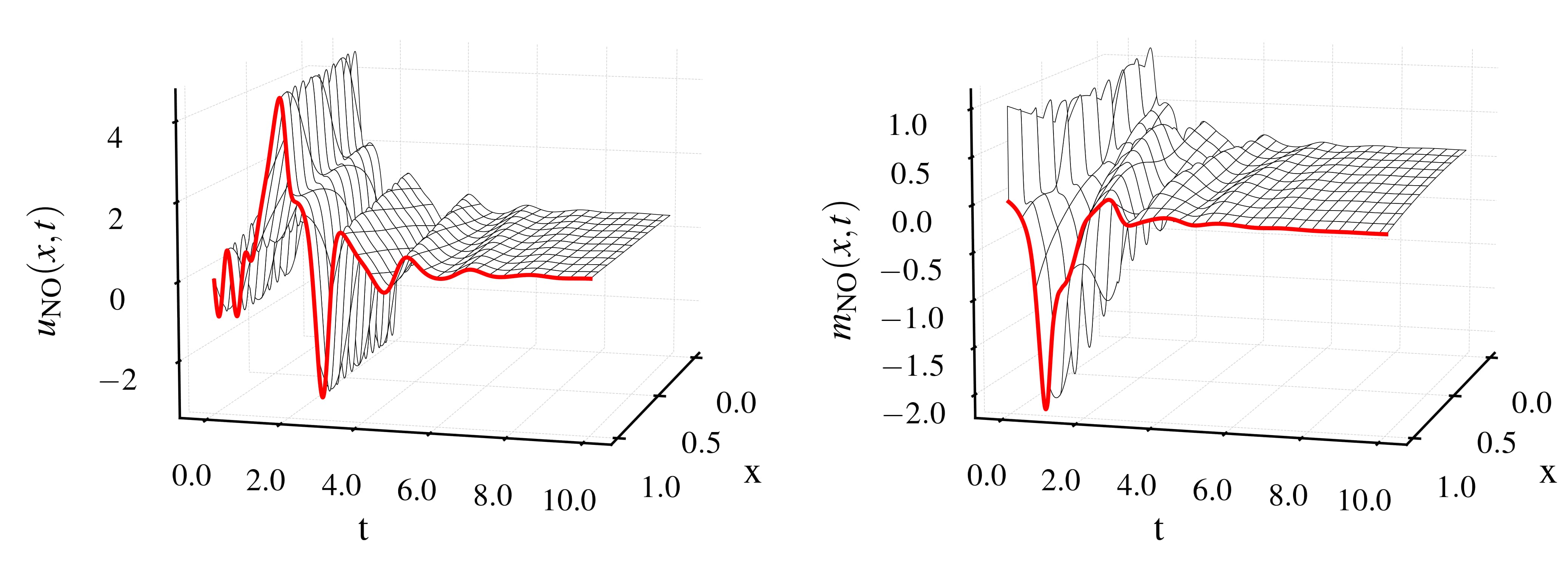}
}
% 第二行：2个居中子图
\subfigure[$u_{RL}(x, t)$ and $m_{RL}(x, t)$ for $\omega_0=2$]{
\includegraphics[width=0.48\textwidth]{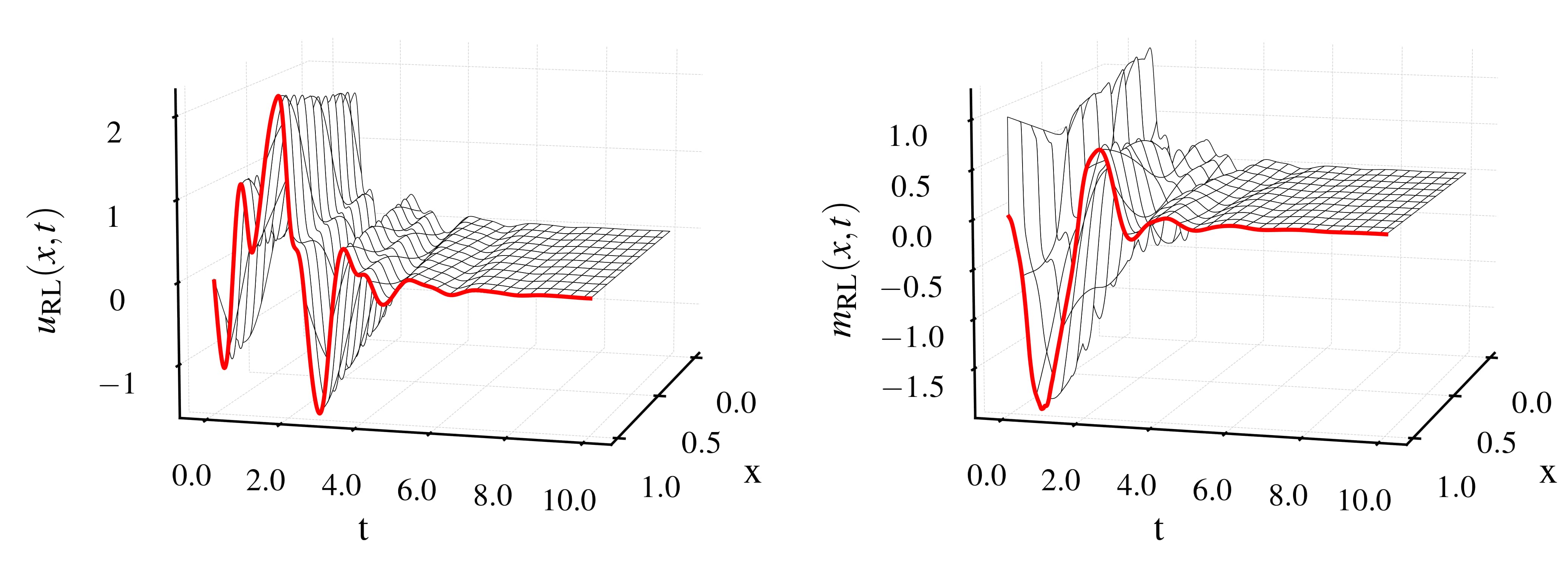}
}
\subfigure[$u_{RL}(x, t)$ and $m_{RL}(x, t)$ for $\omega_0=10$]{
\includegraphics[width=0.48\textwidth]{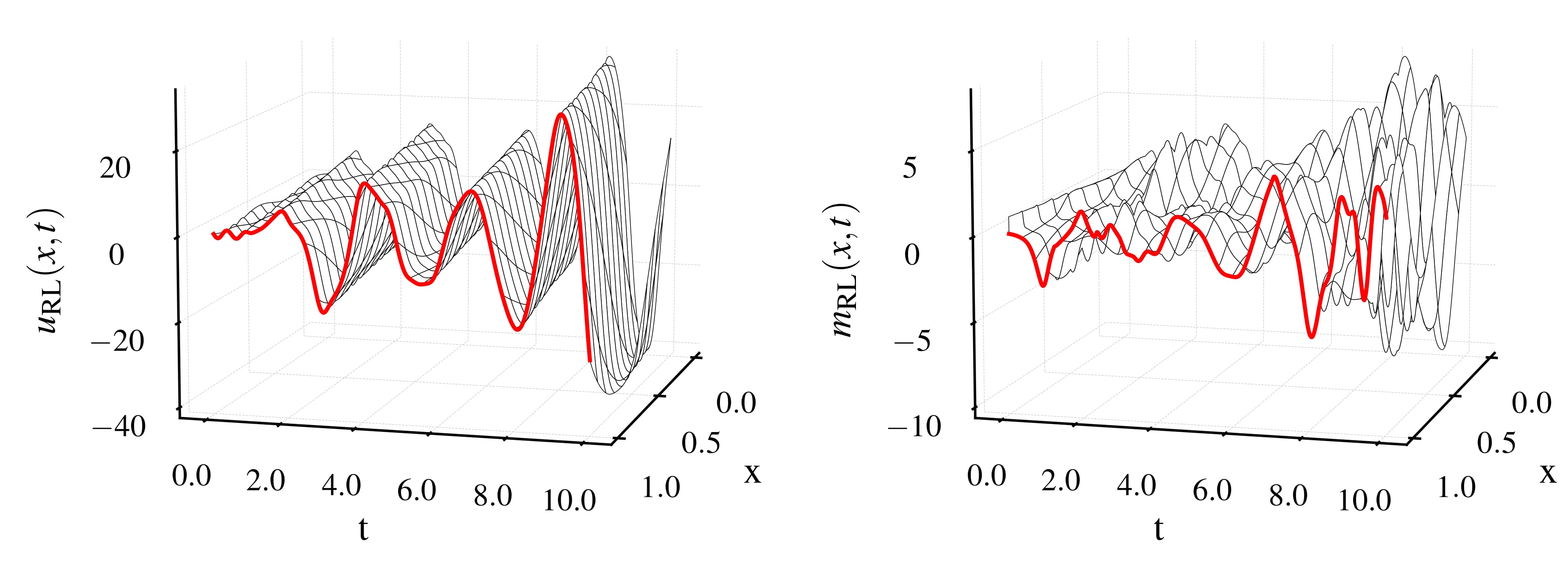}
}
\caption{Comparison of system stability between NO-based adaptive control and RL under different initial conditions \eqref{initial_w0}. (a) State evolution under NO-based adaptive control when the parameter of initial condition \eqref{initial_w0} is $\omega_0=2$. (b) State evolution under NO-based adaptive control when the parameter of initial condition \eqref{initial_w0} is $\omega_0=10$. (c) State evolution under RL when the parameter of initial condition \eqref{initial_w0} is $\omega_0=2$. (d) State evolution under RL when the parameter of initial condition \eqref{initial_w0} is $\omega_0=10$.}
\label{Comparison with RL}
\end{figure*}

%-------------ARZ-----------------------------------

\subsection{Application Simulation of the ARZ Traffic System}
\par \textit{A. NO-based adaptive controller} 
\par Following the steps in the section \ref{sec_Nominal}, we can obtain the adaptive controller for ARZ traffic system \eqref{eq:ARZ_1} as follows
\begin{align}\label{U_ARZ}
	U(t)=&\int_{0}^{1}\breve{K}^{u_1}(1,\xi,t){u}_1(\xi,t)d\xi + \int_{0}^{1}\breve{K}^{m_1}(1,\xi,t){m}_1(\xi,t)d\xi ,
\end{align}
with the parameter update law
\begin{equation}
	\label{eq:lawc}
	\hat{c}_{t}(x,t)  =\operatorname{Proj}_{\bar{c}}\left\{\gamma_{3}  e^{\gamma x}\varepsilon_1(x, t){{u}_1}(x, t) , \hat{c}(x,t)\right\}, 
\end{equation}
where 
\begin{equation*}
	\epsilon_1 (x,t)= m_1(x,t) - \hat{m}_1(x,t),
\end{equation*}
and the kernels satisfy the following kernel functions
\begin{align*}
	(\gamma p^* - {v}^*) \breve{K}_{x}^{u_1}(x, \xi,t)  =&{v}^* \breve{K}_{\xi}^{u_1}(x, \xi,t)\nonumber \\
	&+\hat{c}(x,t) \breve{K}^{m_1}(x, \xi,t), \\
	(\gamma p^* - {v}^*) \breve{K}_{x}^{m_1}(x, \xi,t)  =&-(\gamma p^* - {v}^*) \breve{K}_{\xi}^{m_1}(x, \xi,t), \\
	\breve{K}^{u_1}(x, x,t)  =&-\frac{\hat{c}(x,t)}{\gamma p^*}, \\
	\breve{K}^{m_1}(x, 0,t)  =&\dfrac{{v}^*}{\gamma p^*-{v}^*}r\breve{K}^{u_1}(x, 0,t).
\end{align*}
According to the approximation of NO in Theorem \ref{th:NO}, we get the NO-based adaptive controller 
\begin{align}\label{U_ARZ_NO}
	U(t)=&\int_{0}^{1}\hat{K}^{u_1}(1,\xi,t){u}_1(\xi,t)d\xi + \int_{0}^{1}\hat{K}^{m_1}(1,\xi,t){m}_1(\xi,t)d\xi.
\end{align}
\begin{table}[t]
\centering
\resizebox{\columnwidth}{!}{%
\begin{tabular}{lcccc}
\hline
\textbf{Method} 
& \textbf{\begin{tabular}[c]{@{}c@{}}Average \\ Computation Time (s)\end{tabular}} 
& \multicolumn{2}{c}{\textbf{MSE \%}} \\
\cline{3-4}
& & \textbf{Density $\rho$} & \textbf{ Velocity $v$} \\
\hline
Nominal Adaptive Controller      & $1.51$    & $0$       & $0$ \\
NO-based Adaptive Controller     & $0.043$   & $0.021$   & $0.011$ \\
\hline
\end{tabular}%
}
\caption{\textcolor{black}{Computation time and mean square errors (MSE) of density and velocity for the nominal and NO-based adaptive controllers in traffic control.}}
\label{tab2}
\end{table}

\par \textit{B. Simulation results }
\par Then, we analyze the performance of the proposed NO-based adaptive control law for the ARZ traffic PDE system through simulations on a L=600m road over T=300s. The parameters are set as follows: free-flow velocity \(v_m = 40\) m/s, maximum density \(\rho_m = 160\) veh/km, equilibrium density \(\rho^* = 120\) veh/km, driver reaction time \(\tau = 60\) s. Let $\gamma=1$. Initial conditions are sinusoidal inputs \(\rho(x,0) = \rho^* + 0.1\sin(\frac{3\pi x}{L})\rho^*\) and \(v(x,0) = v^* - 0.01\sin(\frac{3\pi x}{L})v^*\) to mimic stop-and-go traffic. \textcolor{black}{Recent advances in traffic sensing technologies, such as connected vehicles (CVs), loop detectors, and roadside sensors, provide increasingly dense and accurate measurements.}
To generate a sufficient dataset for training, we use 10 different \(c(x)\) functions with \(\tau \in U[50,70]\) and simulate the resulting PDEs under the adaptive controller for \(T=300\) seconds. We sub-sample each \((c, \hat{K}^{u_1}, \hat{K}^{m_1})\) pair every 0.1 seconds, resulting in a total of 30,000 distinct \((c, \hat{K}^{u_1}, \hat{K}^{m_1})\) pairs for training the NO. 
Using the trained NO, we simulate with the same parameters. Figure \ref{ARZ_openloop} shows the ARZ system is open-loop unstable. Figures \ref{rho_v} show the density and velocity of ARZ traffic system. The blue line indicates the initial condition, whereas the red line represents the boundary condition of the system.
The results indicate that both the NO-based adaptive method and the adaptive backstepping control method effectively stabilize the transportation system. The traffic density and velocity converge to the equilibrium values of \(\rho^* = 120\) veh/km and \(v^* = 36\) m/s, respectively. 
\textcolor{black}{As shown in Figure 10, the boundary control input constructed using the DeepONet-based kernels achieves stabilization performance comparable to that of the exact controller, indicating that the learned kernels are sufficiently accurate for practical traffic control applications and effectively alleviate traffic congestion.}
The maximum error does not exceed 10$\%$. The estimated parameter \(\hat{c}\) is shown in Figure \ref{c_Comparison}. 

\par Table \ref{tab2} presents the computation times for both the nominal adaptive controller and the NO-based adaptive controller. As the baseline result, the nominal adaptive control method is compared with the NO-based adaptive control method. Notably, the NO-based adaptive control method not only achieves significantly faster average computation times but also maintains superior accuracy with lower mean square errors. These advantages of the NO-based adaptive control method not only enhance computational efficiency but also make it highly suitable for real-time traffic system applications. The NO method’s efficiency and accuracy represent a substantial advancement, promising more effective and scalable traffic control strategies in practical scenarios.

\begin{figure*}
\centering
\includegraphics[width=14cm]{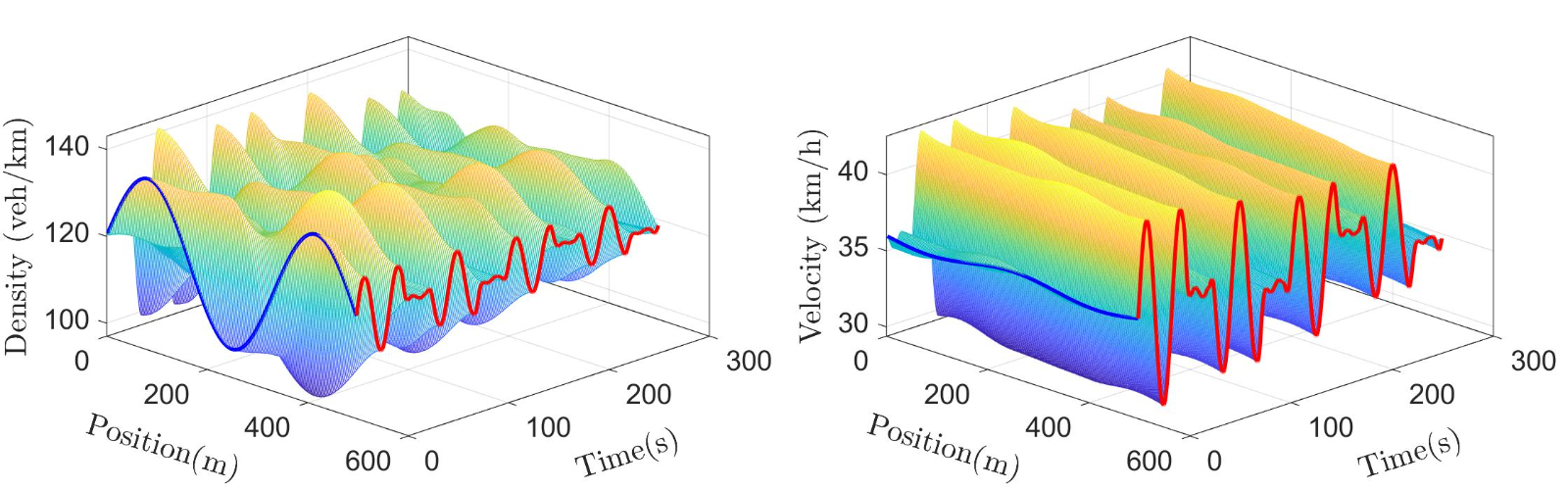}
\caption{Density and velocity evolution of open-loop ARZ traffic system }
\label{ARZ_openloop}
\end{figure*}

\begin{figure*}[ht]
\centering
% 第一行：3个子图
\subfigure[$\rho(x,t)$]{
\includegraphics[width=0.3\textwidth]{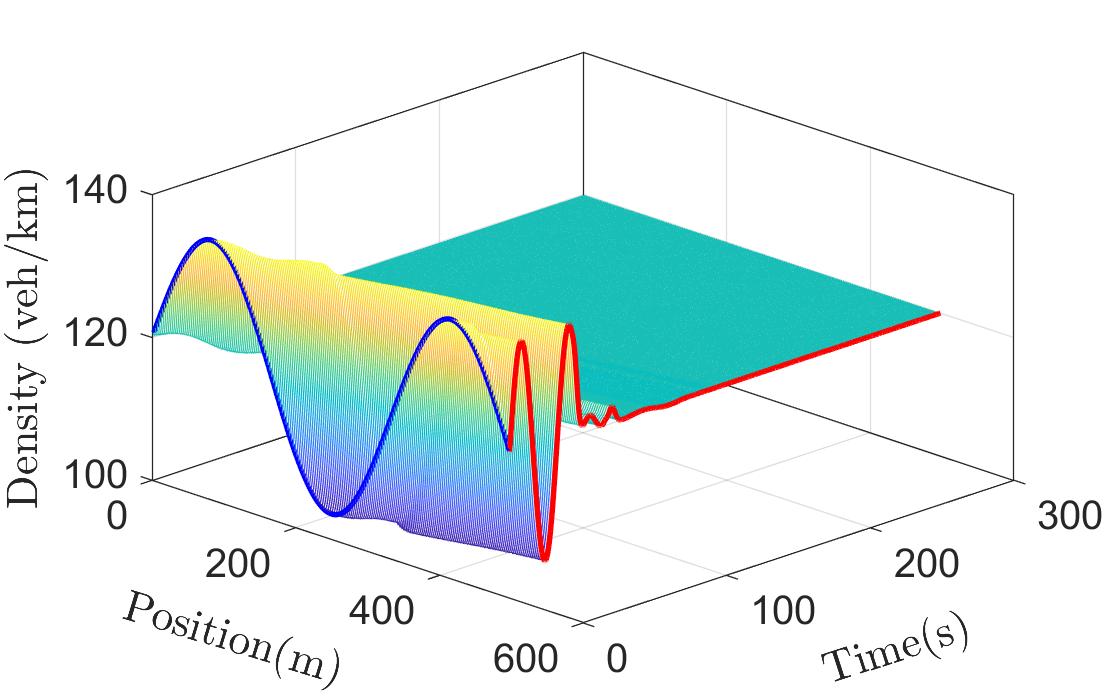} % 替换为你的图片路径
}
\subfigure[$\rho_{NO}(x,t)$]{
\includegraphics[width=0.3\textwidth]{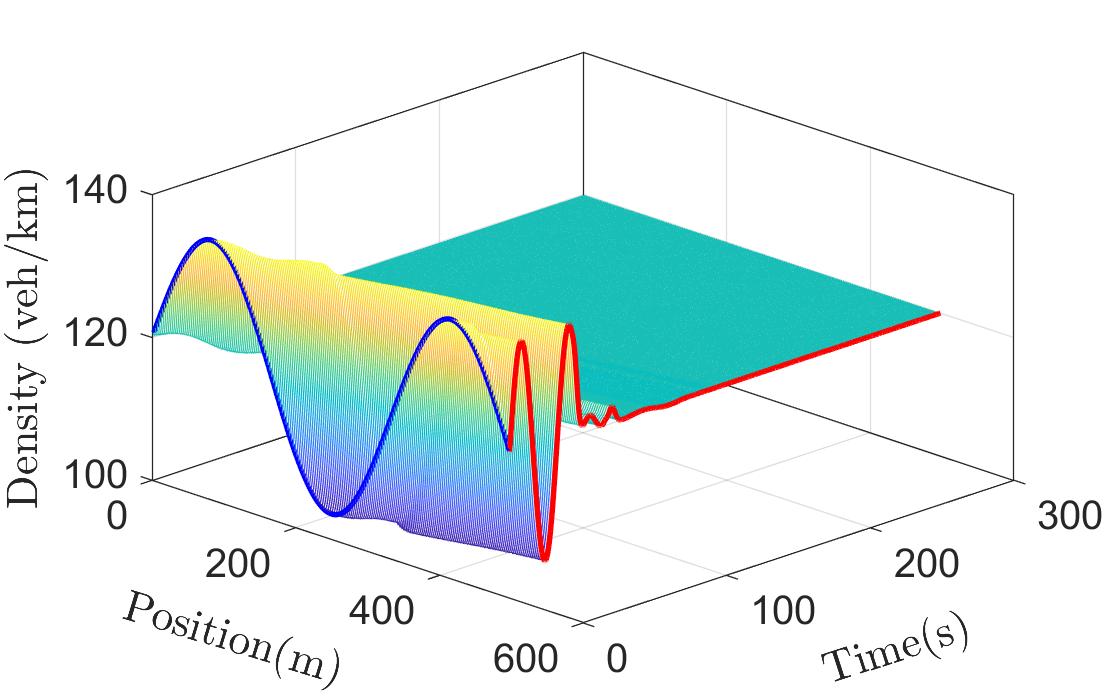}
}
\subfigure[$\|\rho(x,t)-\rho_{NO}(x,t)\|_{L^2}/\rho^*$]{
\includegraphics[width=0.3\textwidth]{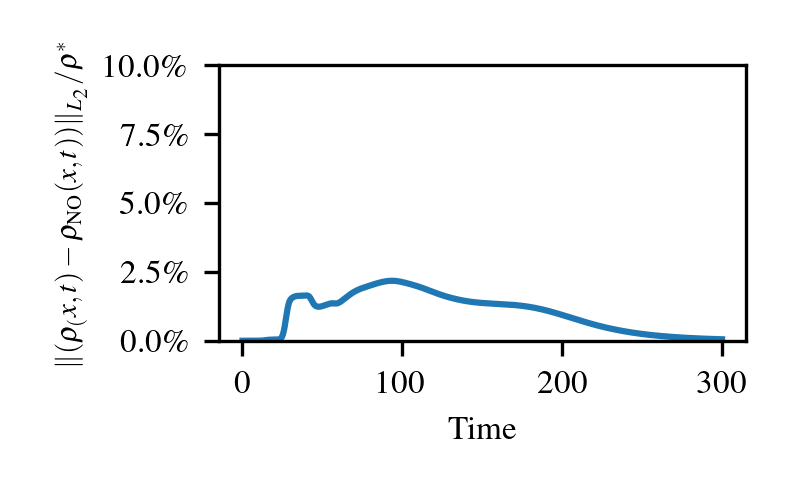}
}
% 第二行：3个居中子图
\subfigure[$v(x,t)$]{
\includegraphics[width=0.3\textwidth]{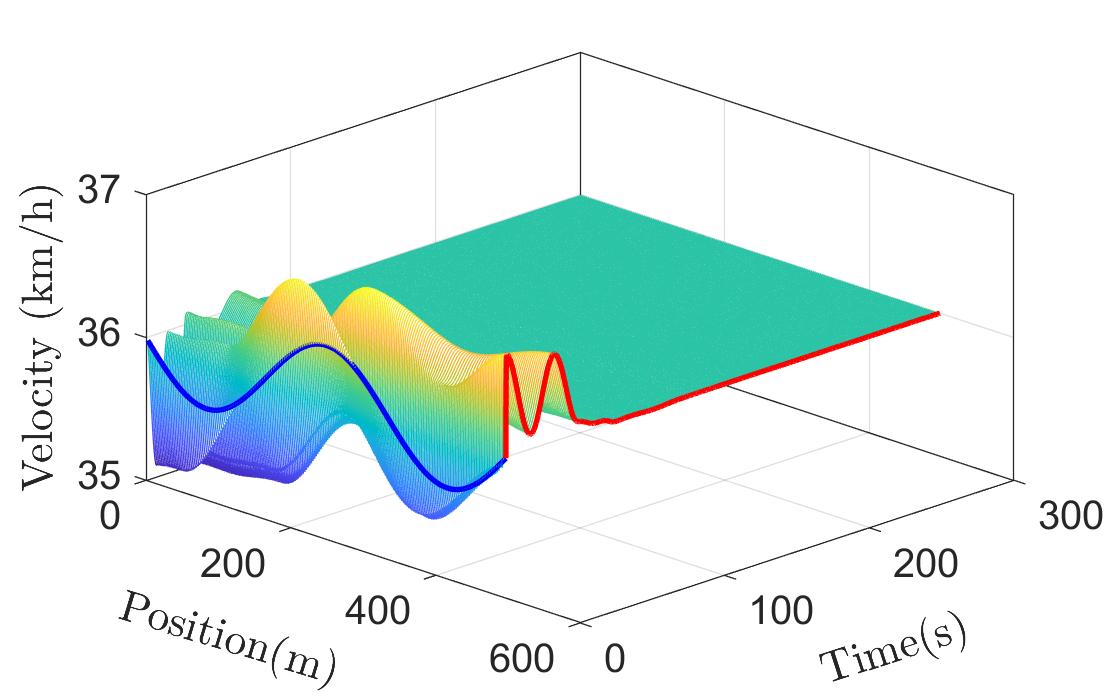} % 替换为你的图片路径
}
\subfigure[$v_{NO}(x,t)$]{
\includegraphics[width=0.3\textwidth]{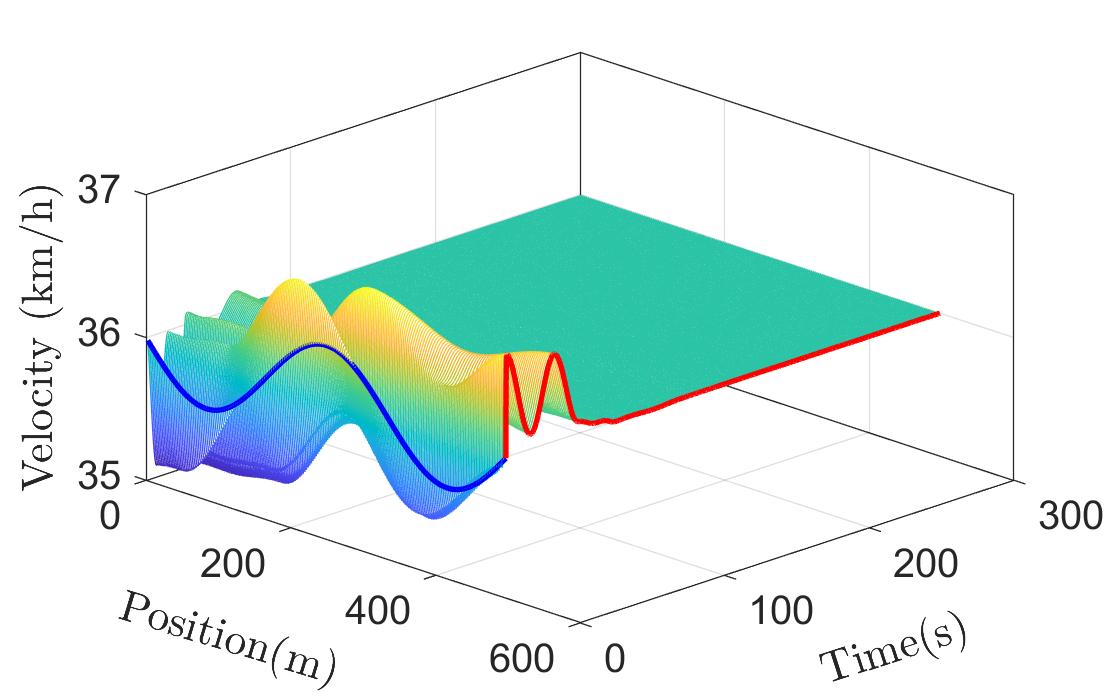}
}
\subfigure[$\|v(x,t)-v_{NO}(x,t)\|_{L^2}/v^*$]{
\includegraphics[width=0.3\textwidth]{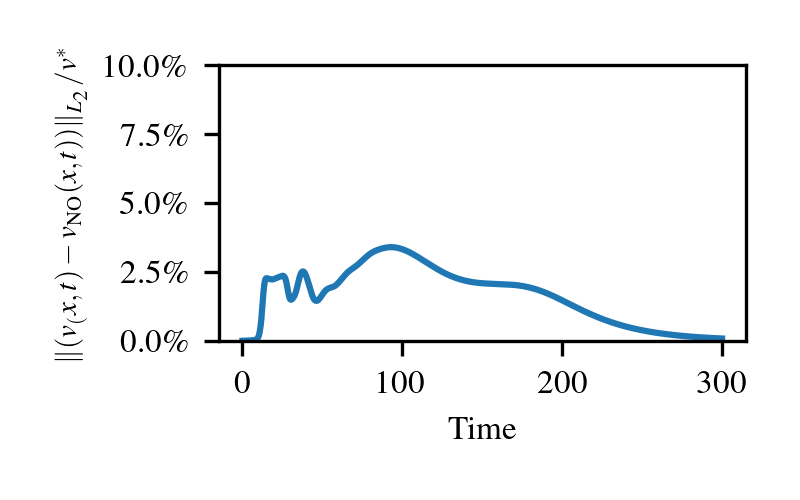}
} 
\caption{Simulation of density and velocity with feedback controller \eqref{U_ARZ} and \eqref{U_ARZ_NO}. The left columns of the first and second rows show close-loop system states $\rho(x,t)$, $v(x,t)$. The middle columns of the first and second rows show close-loop system states $\rho_{NO}(x,t)$, $v_{NO}(x,t)$. The left columns of the first and second rows show the relative $L^2$ error between  $\rho(x,t)$ and  $\rho_{NO}(x,t)$, and between ${v}(x,t)$ and  $v_{NO}(x,t)$, respectively. }
\label{rho_v}
\end{figure*}

\begin{figure*}
\centering
\includegraphics[width=13cm]{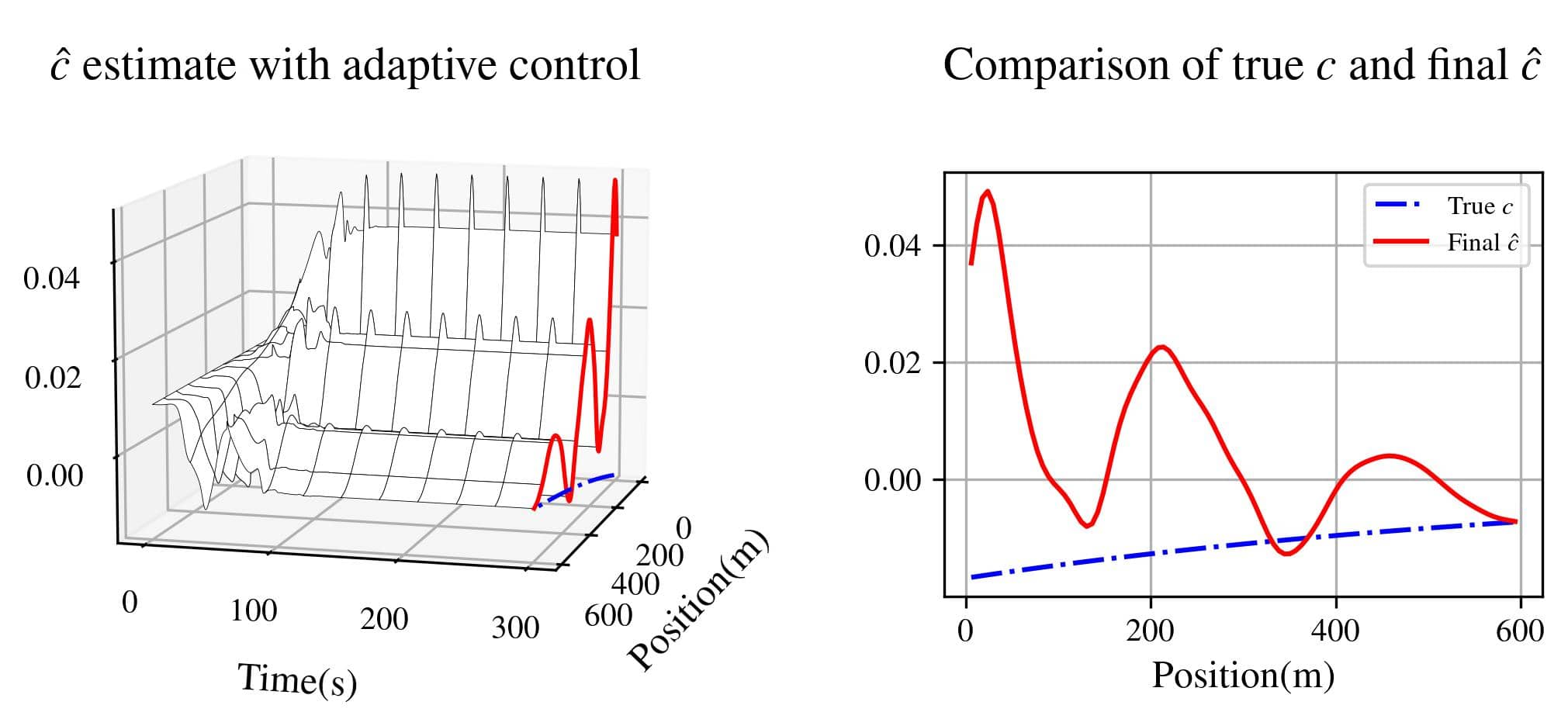}
\caption{Parameter estimation of $\hat{c}$ in ARZ traffic system
and comparision of true $c$ and $\hat{c}$}
\label{c_Comparison}
\end{figure*}

\section{Conclusion}\label{sec_Conclusion}
We develop a NO-based adaptive boundary control design for a 2$\times$2 linear first-order hyperbolic system. Compared with the previous studies \cite{lamarque2025adaptive}, \cite{bhan2025adaptive} that primarily focused on approximating a single kernel PDE, this work accelerates the computation of 2$\times$2 coupled Goursat-form PDEs. In this paper, the DeepONet is used to learn the adaptive control gains for stabilizing the traffic PDE system, and it is shown that under the DeepONet-approximated kernels the stabilization of 2$\times$2 hyperbolic PDEs can still be achieved with significant improvement for computational speeds. Experimental results show that compared to traditional numerical solvers, our method improves computational efficiency by two orders of magnitude. Additionally, compared with RL, the NO-based adaptive control strategy is independent of the system’s initial conditions, making it more robust for rapidly changing traffic scenarios. Our method significantly accelerates the process of obtaining adaptive controllers in PDE systems, greatly improving the real-time applicability of adaptive control strategies for mitigating traffic congestion. In the future, we will incorporate real traffic data into the training of the neural operator.

\appendix
\section{Proof of Lemma \ref{lemma:Identifier_bound}}
The proof of this lemma follows a similar approach to the proof of Lemma 9.1 in \cite{anfinsen2019adaptive}. Property \eqref{properties_1} follows trivially from projection in \eqref{eq:lawc12} and Lemma A.1 in \cite{anfinsen2019adaptive}.
The result can be easily obtained using the following Lyapunov function candidate:
\begin{align}
V_1(t) =& V_2(t)+\gamma_1^{-1}\|\tilde{c}_{1}\|^2+\gamma_2^{-1}\|\tilde{c}_{2}\|^2+\gamma_3^{-1}\|\tilde{c}_{3}\|^2 \nonumber\\
&+\gamma_4^{-1}\|\tilde{c}_{4}\|^2 +\dfrac{\lambda}{2\gamma_5}\tilde{r}^2(t), \label{eq:V_1}
\end{align}
where
\begin{align*}
\textcolor{black}{
V_2(t)=\int_{0}^{1}e^{-\gamma x}e_1^2(x,t)dx + \int_{0}^{1} e^{\gamma x}e_2^2(x,t)dx.} 
\end{align*}
Computing the time derivative of \eqref{eq:V_1} along \eqref{eq:e_1}-\eqref{eq:e_4} as
\begin{align*}
\dot{V}_{1}(t)= & 2 \int_{0}^{1} \Big( e^{-\gamma x} e_1(x, t) e_{1t}(x, t)+ e^{\gamma x} e_2(x, t) e_{2t}(x, t) \Big) d x \nonumber\\
&+2 \gamma_{1}^{-1}\int_{0}^{1} \tilde{c}_{1}(x,t)\tilde{c}_{1t}(x,t)dx \nonumber\\
&+2 \gamma_{2}^{-1}\int_{0}^{1} \tilde{c}_{2}(x,t)\tilde{c}_{2t}(x,t)dx  \nonumber\\
&+2 \gamma_{3}^{-1}\int_{0}^{1} \tilde{c}_{3}(x,t)\tilde{c}_{3t}(x,t)dx \nonumber\\
&+2 \gamma_{4}^{-1}\int_{0}^{1} \tilde{c}_{4}(x,t)\tilde{c}_{4t}(x,t)dx  +\lambda \gamma_{5}^{-1}  \tilde{r}(t) \dot{\tilde{r}}(t). %\nonumber\\
\end{align*}
Substituting into the dynamics \eqref{eq:e_1}-\eqref{eq:e_4} and integrating by parts, we obtain
\begin{align*}
\dot{V}_{1}(t)= & -\lambda e^{-\gamma} e_1^{2}(1, t)+\lambda e_1^{2}(0, t)-\lambda \gamma \int_{0}^{1} e^{-\gamma x} e_1^{2}(x, t) d x  \nonumber\\
&+2 \int_{0}^{1} e^{-\gamma x} e_1(\tilde{c}_{1}u+\tilde{c}_{2}m) d x \nonumber\\
&-2 \rho \int_{0}^{1} e^{-\gamma x} e_1^{2}(x, t)\|\varpi(t)\|^{2} d x \nonumber\\
&-\mu e_2^{2}(0, t)-\mu \gamma \int_{0}^{1} e^{\gamma x} e_2^{2}(x, t) d x \nonumber\\
&+2 \int_{0}^{1} e^{\gamma x} e_2(\tilde{c}_{3}u+\tilde{c}_{4}m) d x \nonumber\\
&-2 \rho \int_{0}^{1} e^{\gamma x} e_2^{2}(x, t)\|\varpi(t)\|^{2} d x\nonumber\\
&+2 \gamma_{1}^{-1}\int_{0}^{1} \tilde{c}_{1}(x,t)\tilde{c}_{1t}(x,t)dx \nonumber\\
&+2 \gamma_{2}^{-1}\int_{0}^{1} \tilde{c}_{2}(x,t)\tilde{c}_{2t}(x,t)dx  \nonumber\\
&+2 \gamma_{3}^{-1}\int_{0}^{1} \tilde{c}_{3}(x,t)\tilde{c}_{3t}(x,t)dx \nonumber\\
&+2 \gamma_{4}^{-1}\int_{0}^{1} \tilde{c}_{4}(x,t)\tilde{c}_{4t}(x,t)dx  +\lambda \gamma_{5}^{-1} \tilde{r}(t) \dot{\tilde{r}}(t) .
\end{align*}     	
Inserting the adaptive laws \eqref{eq:lawc12}. By using the property \eqref{eq:projection_proper}, we have
\begin{align}
    \tilde{c}_1(x,t)\tilde{c}_{1t}(x,t)&=-\tilde{c}_1(x,t)\hat{c}_{1t}(x,t)\nonumber\\
    &=-\tilde{c}_{1}(x,t) \operatorname{Proj}_{\bar{c}_{1}} (\gamma_1 e^{-\gamma x}e_1 u,\hat{c}_1(x,t)) \nonumber\\
    &\leq -\tilde{c}_{1}(x,t)\gamma_1 e^{-\gamma x}e_1 u.
\end{align} Similarly for $\tilde{c}_2,\tilde{c}_3, \tilde{c}_4$, and $\tilde{r}$.
Then we have 
\begin{equation}
\begin{split}
\dot{V}_{1}(t) \leq & -\lambda e^{-\gamma} e_1^{2}(1, t)+\lambda e_1^{2}(0, t)-\lambda \gamma \int_{0}^{1} e^{-\gamma x} e_1^{2}(x, t) d x \\
& -2 \rho \int_{0}^{1} e^{-\gamma x} e_1^{2}(x, t)\|\varpi(t)\|^{2} d x\\
&-\mu e_2^{2}(0, t)-\mu \gamma \int_{0}^{1} e^{\gamma x} e_2^{2}(x, t) d x \\
& -2 \rho \int_{0}^{1} e^{\gamma x} e_2^{2}(x, t)\|\varpi(t)\|^{2} d x-\lambda \tilde{r}(t) e_1(0, t) m(0, t).\label{eq:A6}
\end{split}     	
\end{equation}
From the boundary condition \eqref{eq:e_3}, we have
\begin{equation}\label{eq:e(0,t)}
e_1(0, t)-\tilde{r}(t) m(0, t)=-e_1(0, t) m^{2}(0, t).
\end{equation}
\textcolor{black}{For the second and last terms in \eqref{eq:A6}, substituting \eqref{eq:e(0,t)}, we have}
\begin{align}
&\textcolor{black}{\lambda e_1^2(0,t) -\lambda \tilde{r}(t) e_1(0,t)m(0,t)} \nonumber \\
&\textcolor{black}{= \lambda e_1(0,t)\left(e_1(0,t) - \tilde{r}(t)m(0,t)\right)} \nonumber \\
&\textcolor{black}{= -\lambda e_1^2(0,t) m^2(0,t) .}\label{eq:A8}
\end{align}
By substituting \eqref{eq:A8}, we obtain
%\begin{equation}
\begin{align}\label{eq:dotV1}
\dot{V}_{1}(t) \leq & -\lambda e^{-\gamma} e_1^{2}(1, t)-\lambda e_1^{2}(0, t) m^{2}(0, t)-\lambda \gamma e^{-\gamma}\|e_1(t)\|^{2} \nonumber\\
& -2 \rho e^{-\gamma}\|e_1(t)\|^{2}\|\varpi(t)\|^{2}-\mu e_2^{2}(0, t) \nonumber\\
& -\mu \gamma\|e_2(t)\|^{2}-2 \rho e^{\gamma}\|e_2(t)\|^{2}\|\varpi(t)\|^{2}.
\end{align}     	
%\end{equation}
From \eqref{eq:dotV1}, we obtain that $V_1$ is bounded. By the definitions of $V_1$ and $V_2$, it follows that $\|e_1\|,\|e_2\| \in L^{\infty}$. When \eqref{eq:dotV1} is integrated over time from zero to infinity, we conclude the results that $\|e_1\|,\|e_2\| \in L^{2}$, \eqref{properties_3}, and $e_1(1,\cdot),e_2(0,\cdot), |e_1(0,\cdot)m(0,\cdot)| \in  L^{2}$. From above results and the adaptive laws \eqref{eq:lawc12}, we derive that \eqref{properties_4}. we choose the Lyapunov function candidate
$
V_{3}(t)=\frac{1}{2} \gamma_{5}^{-1} \tilde{r}^{2}(t),$
and use the property \eqref{eq:projection_proper}, we find
\begin{equation}\label{eq:dotV3}
\dot{V}_{3}(t) \leq-\tilde{r}(t) e_1(0, t) m(0, t) \leq-\frac{\tilde{r}^{2}(t) m^{2}(0, t)}{1+m^{2}(0, t)} .
\end{equation}
This implies that $V_3$ is upper-bounded, and hence we have $V_3 \in L^{\infty}$. By integrating \eqref{eq:dotV3} from zero to infinity, we obtain  \eqref{properties_5}. Using \eqref{eq:e(0,t)} and \eqref{eq:e_3}, we derive that
\begin{equation}
\begin{split}
e_1^{2}(0, t) & =e_1(0, t)\left(\tilde{r}(t) m(0, t)-e_1(0, t) m^{2}(0, t)\right) \\
& =\frac{\tilde{r}^{2}(t) m^{2}(0, t)}{1+m^{2}(0, t)}-e_1^{2}(0, t) m^{2}(0, t),
\end{split}     	
\end{equation}
and from $|e_1(0,\cdot)m(0,\cdot)| \in  L^{2}$ and \eqref{properties_5}, we have $e_1(0,\cdot) \in  L^{2}$.

\bibliographystyle{elsarticle-num}

\bibliography{refDataBase} 

\end{document}